\documentclass[hidelinks,onefignum,onetabnum]{siamart251104}

\usepackage{lipsum}
\usepackage{amsfonts}
\usepackage{graphicx}
\usepackage{epstopdf}
\usepackage{algorithmic}
\usepackage{booktabs}
\usepackage{caption}

\ifpdf
  \DeclareGraphicsExtensions{.eps,.pdf,.png,.jpg}
\else
  \DeclareGraphicsExtensions{.eps}
\fi

\newsiamremark{remark}{Remark}
\newsiamremark{hypothesis}{Hypothesis}
\crefname{hypothesis}{Hypothesis}{Hypotheses}
\newsiamthm{claim}{Claim}
\newsiamremark{fact}{Fact}
\crefname{fact}{Fact}{Facts}
\newsiamremark{assumption}{Assumption}

\headers{Stochastic versus Deterministic}{R. Li, J. Xu, and W. Xing}

\title{Stochastic versus Deterministic in Stochastic Gradient Descent\thanks{Submitted to the editors DATE.
\funding{The work of Runze Li and Wenxun Xing was supported by the National Natural Science Foundation of China Grant No. 11771243. The work of Jintao Xu was supported in part by PolyU postdoc matching fund scheme of the Hong Kong Polytechnic University Grant No. 1-W35A, and Huawei's Collaborative Grants ``Large scale linear programming solver'' and ``Solving large scale linear programming models for production planning''. }}}

\author{Runze Li \thanks{Department of Mathematical Sciences, Tsinghua University, Beijing 100084, China (\email{li-rz22@mails.tsinghua.edu.cn}, \email{wxing@tsinghua.edu.cn}).}
\and Jintao Xu\thanks{Corresponding author. Department of Applied Mathematics, The Hong Kong Polytechnic University, Hong Kong, China (\email{xujtmath@163.com}).}
\and Wenxun Xing\footnotemark[2]}

\usepackage{amsopn}

\ifpdf
\hypersetup{
  pdftitle={Stochastic versus Deterministic in Stochastic Gradient Descent},
  pdfauthor={R. Li, J. Xu, and W. Xing}
}
\fi

\allowdisplaybreaks[1]

\begin{document}

\maketitle

\begin{abstract}
This paper theoretically reanalyzes the convergence of the mini-batch stochastic gradient descent (SGD) for a structured minimization problem involving a finite-sum function with its gradient being stochastically approximated, and an independent term with its gradient being deterministically computed.  Rather than collapsing this problem into a standard finite-sum formulation and treating all components uniformly, we study it from a stochastic versus deterministic viewpoint and focus on how these two gradient computations affect mini-batch SGD differently. The step size, the convergence rate, and the radius of the convergence region depend asymmetrically on the characteristics of the two components, which shows the distinct impacts of stochastic approximation versus deterministic computation in the mini-batch SGD. Based on this, we show that our analysis yields a faster convergence rate and a smaller radius of the convergence region. Moreover, an even better convergence rate can be obtained when the independent term endows the objective function with sufficient strong convexity. Also, the convergence rate of our algorithm in expectation approaches that of the classic gradient descent when the batch size increases. Numerical experiments are conducted to support the theoretical analysis as well.
\end{abstract}



\section{Introduction}

In machine learning, some problems can be formulated as follows \cite{shalev2014understanding}
\begin{equation}\label{eq:opt1}
	\min _{\boldsymbol{x} \in \mathbb{R}^d} \quad \psi(\boldsymbol{x}) := F(\boldsymbol{x})+h(\boldsymbol{x}),
\end{equation}
where $F(\boldsymbol{x}):=\frac{1}{n} \sum_{i=1}^{n} f_i(\boldsymbol{x})$ is the empirical risk that depends on a dataset with a large size $n$, while $h(\boldsymbol{x})$ represents another objective term, such as a regularizer, which is generally independent of the data size. This formulation is referred to as the structured optimization problem in this paper. Although \eqref{eq:opt1} allows  $f_i(\boldsymbol{x})$, $i=1,\dots,n$ and $h(\boldsymbol{x})$ to be nonsmooth \cite{bolte2021conservative,metel2021stochastic}, we focus on the setting in which both are smooth. The detailed background of \eqref{eq:opt1} is provided in \Cref{sec:2}.

Early research addressed \eqref{eq:opt1} with deterministic optimizers.  A classic approach is the gradient descent, whose update is
\begin{equation}
	\boldsymbol{x}_{k+1} = \boldsymbol{x}_{k} - \eta_k \left(\frac{1}{n}\sum_{i=1}^n \nabla f_i(\boldsymbol{x}_k) +\nabla h(\boldsymbol{x}_k)\right).
\end{equation}
This algorithm and its variants have been studied widely in theory and practice \cite{nesterov2018lectures,drori2014performance}.  As the data size~$n$ grows, computing the full gradient $\tfrac{1}{n}\sum_{i=1}^{n}\nabla f_i(\boldsymbol{x})$ becomes expensive, leading to increased use of stochastic algorithms.  Among them, the stochastic gradient descent (SGD), which is based on the stochastic approximation proposed by Robbins and Monro \cite{robbins1951stochastic}, is widely used.  Its core idea is to estimate the gradient at each iteration with a single sample or a small batch, thereby cutting the per-iteration cost. Beyond SGD, many stochastic optimization methods have been proposed to take advantage of the specific problem structures.  Representative examples include proximal SGD \cite{xiao2014proximal}, accelerated stochastic approximation \cite{ghadimi2012optimal}, and randomized block-coordinate descent \cite{richtarik2014iteration,fercoq2015accelerated}. These algorithms leverage structural properties of the objective function and are supported by sound convergence analysis.

Despite the existence of various algorithms, vanilla SGD remains a popular choice. As shown in \cite{zhang2023dive}, it is widely applied to large-scale instances of \eqref{eq:opt1} in practice and continues to attract research interest. Its popularity stems from the simplicity, computational efficiency, and seamless integration into mainstream machine learning frameworks such as TensorFlow \cite{abadi2016tensorflow} and PyTorch \cite{paszke2019pytorch}.

Therefore, we study the performance of the following mini-batch SGD for the problem \eqref{eq:opt1}.  
The iterative scheme is given by  
\begin{equation}\label{eq:mini_batch_sgd}
	\boldsymbol{x}_{k+1} = \boldsymbol{x}_{k} - \eta_k \Bigg(\frac{1}{B}\sum \limits_{j=1}^{B} \frac{1}{\textstyle np_{\scriptscriptstyle \xi_{k,j}}}\nabla f_{\xi_{k,j}}(\boldsymbol{x}_k)+\nabla h(\boldsymbol{x}_k)\Bigg),
\end{equation}
where $B$ is the batch size, and $\xi_{k,j}$, $j=1,\ldots,B$, are independent random variables taking values in $\{1,\ldots,n\}$, with $\mathbb{P}(\xi_{k,j}=i)=p_i>0$ for $i=1,\ldots,n$ and $j=1,\ldots,B$. The mini-batch gradient in \eqref{eq:mini_batch_sgd} is an unbiased estimator of the full gradient  
$\nabla F(\boldsymbol{x})=\frac{1}{n}\sum_{i=1}^{n}\nabla f_i(\boldsymbol{x})$, that is,
$\mathbb{E}\Big[\frac{1}{B}\sum_{j=1}^{B} \tfrac{1}{\textstyle np_{\scriptscriptstyle \xi_{k,j}}}\nabla f_{\xi_{k,j}}(\boldsymbol{x})\Big]=\frac{1}{n}\sum_{i=1}^{n} \nabla f_i(\boldsymbol{x})$, where the expectation is taken over $\xi_{k,1},\ldots,\xi_{k,B}$.  
As shown in \eqref{eq:mini_batch_sgd}, when applying the mini-batch SGD to solve the problem \eqref{eq:opt1}, the gradient of $F(\boldsymbol{x})$ is estimated stochastically (by sampling from $\{\nabla f_i(\boldsymbol{x})\}_{i=1}^n$), while the gradient of $h(\boldsymbol{x})$ is computed deterministically. This asymmetric treatment of $F(\boldsymbol{x})$ and $h(\boldsymbol{x})$ motivates the following question:

\vspace{1mm}
\textit{How do the stochastic approximation of  $\nabla F(\boldsymbol{x})$ and the deterministic computation of $\nabla h(\boldsymbol{x})$ affect the performance of the mini-batch SGD for the problem \eqref{eq:opt1} differently, and how can this impact be analyzed and explained theoretically?}
\vspace{1mm}

Answering this question serves two purposes. First, it clarifies how the two components of the objective function influence step size selections and the convergence of the algorithm, offering guidance for tuning hyperparameters.  
Second, it offers inspiration for model formulation. For example, in practice, one may consider reassigning terms between $F(\boldsymbol{x})=\tfrac{1}{n}\sum_{i=1}^n f_i(\boldsymbol{x})$ and $h(\boldsymbol{x})$ to better match the behavior of the optimization algorithm.

In the literature we have surveyed, relatively limited attention has been devoted to this issue, possibly for the following reason.  Observe that the problem \eqref{eq:opt1} can be rewritten as
\begin{equation}\label{eq:opt2}
	\min_{\boldsymbol{x} \in \mathbb{R}^d} \quad \psi(\boldsymbol{x}) := \frac{1}{n} \sum_{i=1}^{n} \psi_i(\boldsymbol{x}),
\end{equation}
where $\psi_i(\boldsymbol{x})=f_i(\boldsymbol{x})+h(\boldsymbol{x})$ for $i=1,\ldots,n$.  This reformulation shows that many existing convergence results for SGD applied to \eqref{eq:opt2} can be transferred directly to \eqref{eq:opt1}.  Several studies such as \cite{schmidt2017minimizing} and \cite{nguyen2018sgd} make similar observations for the problems they study.

To clarify how previous work relates to the question we raised, we review two categories of convergence analysis for SGD.  The first category makes explicit assumptions about the stochastic gradients, whereas the second category does not make such assumptions, except for those related to how the gradients are sampled. Although some assumptions in the second category may imply those in the first, the two categories of results are typically stated in different ways, so we discuss them separately. To facilitate the discussion, let $\{\boldsymbol{x}_k\}_{k=1}^{K}$ be the sequence generated by the iteration \eqref{eq:mini_batch_sgd}, and let $\mathcal{F}_{k}$ denote the $\sigma$-field generated by $\{\boldsymbol{x}_1,\ldots,\boldsymbol{x}_{k}\}$ for $k=1,\ldots,K$.  Let $\boldsymbol{G}_k$ and $\boldsymbol{g}_k$ be the stochastic gradients of $\psi(\boldsymbol{x})$ and $F(\boldsymbol{x})$ computed at iteration $k$, respectively. Unless stated otherwise, we consider these stochastic gradients to be unbiased and all convergence results are understood in the sense of expectation. Moreover, we say that $\psi(\boldsymbol{x})$ has Lipschitz continuous gradient with constant $L$  when
$\|\nabla \psi(\boldsymbol{x})-\nabla \psi(\boldsymbol{y})\|_2 \le L \|\boldsymbol{x}-\boldsymbol{y}\|_2$. We say that $\psi(\boldsymbol{x})$ is strongly convex with convexity parameter $\lambda>0$ when $\psi(\boldsymbol{x})-\frac{\lambda}{2}\|\boldsymbol{x}\|_2^2$ is a convex function.

Rakhlin et al. \cite{rakhlin2011making} assumed that the second moment of the stochastic gradient is bounded, that is, there is a constant $M$ such that
$\mathbb{E}[\|\boldsymbol{G}_k\|_2^2 \mid \mathcal{F}_k] \leq M$, $k=1,\dots,K$. They showed that if $\psi(\boldsymbol{x})$ is strongly convex and its gradient is Lipschitz continuous, then choosing the step size as $\eta_k = 1/(\lambda k)$ yields
$\psi(\boldsymbol{x}_k) - \psi(\boldsymbol{x}^*) = \mathcal{O}\bigl(1/k\bigr),$
where $\boldsymbol{x}^*$ is the minimizer of \eqref{eq:opt2} and $\lambda$ is the convexity parameter. Similar results appeared in \cite{nemirovski2009robust}. Lan \cite{lan2020first} studied the stochastic mirror descent, of which SGD is a special case, under the bounded variance assumption $\mathbb{E}\bigl[\|\boldsymbol{G}_k - \nabla \psi(\boldsymbol{x}_k)\|_2^2 \mid \mathcal{F}_k\bigr] \le M$, $k=1,\dots,K$. The author provided convergence results when $\psi(\boldsymbol{x})$ is a non-smooth convex Lipschitz function (using subgradients), a smooth convex function, or a strongly convex function. Bottou et al. \cite{bottou2018optimization} conducted their analysis under a relaxed growth condition. They assumed that there exist constants $M_G$ and $M$ such that
$\mathbb{E}[\|\boldsymbol{G}_k\|_2^2 \mid \mathcal{F}_k] \leq M_G \|\nabla \psi(\boldsymbol{x}_k)\|_2^2 + M$. When $\psi(\boldsymbol{x})$ is strongly convex with Lipschitz continuous gradient and a constant step size $\eta_k \equiv 1/(L M_G)$ is adopted, the sequence $\{\psi(\boldsymbol{x}_k)\}$ converges linearly to a neighborhood of $\psi(\boldsymbol{x}^*)$, where $L$ denotes the Lipschitz constant. They further demonstrated that, under the same conditions, the sequence $\{\psi(\boldsymbol{x}_k)\}$ converges to $\psi(\boldsymbol{x}^*)$ at a rate of $\mathcal{O}(1/k)$ when the step size is chosen as $\eta_k = \frac{\alpha}{\beta + k}$, where $\alpha$ and $\beta$ are suitable hyperparameters. For simplicity, a thorough overview of alternative stochastic gradient assumptions and the convergence analysis of SGD under broader conditions is omitted. For readers seeking more in-depth coverage, please see \cite{lan2020first,bottou2018optimization,wang2023convergence,hu2020biased}. This category of results depends on parameters in the stochastic gradient assumptions. Note that we have $\boldsymbol{G}_k = \boldsymbol{g}_k + \nabla h(\boldsymbol{x}_k)$ for the problem \eqref{eq:opt1}. Directly applying these results to \eqref{eq:opt1} would require treating $\nabla h(\boldsymbol{x}_k)$ as part of the stochastic gradient and thus including $h(\boldsymbol{x})$ in the assumptions. However, since the gradient of $h(\boldsymbol{x})$ is computed deterministically, it is more natural to impose stochastic gradient assumptions only on $\boldsymbol{g}_k$.

For another category of work, we first review studies on specific models of the problem \eqref{eq:opt1}. Zhang \cite{zhang2004solving} considered linear prediction methods with an $\ell_2$ regularizer. It is proved that SGD with a fixed step size converges linearly to a neighborhood of the minimizer. Shalev-Shwartz et al. \cite{shalev2011pegasos} studied the primal support vector machine and showed that with probability at least $1-\delta$,
$\frac{1}{K}\sum_{k=1}^K \psi(\boldsymbol{x}_k) - \psi(\boldsymbol{x}^*) = \mathcal{O}\bigl(\log(K/\delta)/K\bigr)$. Other works considered the problem \eqref{eq:opt2} in the general setting where each $\psi_i(\boldsymbol{x})$ satisfies some smoothness and convexity assumptions. Schmidt and Le Roux \cite{schmidt2013fast} assumed a strong growth condition,
$\max_i \|\nabla \psi_i(\boldsymbol{x})\|_2 \le M_G \|\nabla \psi(\boldsymbol{x})\|_2$.
They showed that if $\psi(\boldsymbol{x})$ is convex with Lipschitz continuous gradient, then with a constant step size $\eta_k \equiv 1/(L M_G^2)$, $\psi(\boldsymbol{x}_k)-\psi(\boldsymbol{x}^*) = \mathcal{O}(1/k)$, where $L$ is the Lipschitz constant. If $\psi(\boldsymbol{x})$ is also strongly convex, the rate is linear. Bach and Moulines \cite{moulines2011non} proved that if $\psi(\boldsymbol{x})$ is strongly convex with Lipschitz continuous gradient, then with a constant step size, the sequence $\{\boldsymbol{x}_k\}$ converges linearly to a neighborhood of $\boldsymbol{x}^*$. They also showed that the rate is $\mathcal{O}(1/k)$ when $\eta_k$ is proportional to $1/k$. Similar results appeared in \cite{nguyen2018sgd} and \cite{gower2019sgd}. This category of results depends on parameters such as Lipschitz constants. When $\psi_i(\boldsymbol{x}) = f_i(\boldsymbol{x})+h(\boldsymbol{x})$ meets these smoothness and convexity assumptions, their results can be applied to the problem \eqref{eq:opt1}. But since $\psi_i(\boldsymbol{x})$ includes both $f_i(\boldsymbol{x})$ and $h(\boldsymbol{x})$, applying these analyses directly treats $f_i(\boldsymbol{x})$ and $h(\boldsymbol{x})$ uniformly, which overlooks the asymmetric roles that stochastic and deterministic gradient computations play in SGD. See \Cref{sec:4} for a detailed discussion.

In this paper, we take the above question as the starting point of our study. 
	Rather than collapsing $\psi(\boldsymbol{x})=F(\boldsymbol{x})+h(\boldsymbol{x})$ into the standard finite-sum
	$\frac{1}{n}\sum_{i=1}^{n}\bigl(f_i(\boldsymbol{x})+h(\boldsymbol{x})\bigr)$ 
	and treating all components uniformly, we preserve the structured decomposition $F(\boldsymbol{x})+h(\boldsymbol{x})$ and analyze the mini-batch SGD according to how the two gradient components are computed in practice: $\nabla F(\boldsymbol{x})$ is stochastically approximated, whereas $\nabla h(\boldsymbol{x})$ is computed exactly. 
	This viewpoint reveals a stochastic-deterministic asymmetry that is hidden in standard SGD analyses. 
	Under the assumptions that $\psi(\boldsymbol{x})=\tfrac{1}{n}\sum_{i=1}^n f_i(\boldsymbol{x})+h(\boldsymbol{x})$ is strongly convex and that each $f_i(\boldsymbol{x})$, $i=1,\dots,n$, and $h(\boldsymbol{x})$ have Lipschitz continuous gradients, our contributions are as follows.

\begin{itemize}
	\item We introduce a stochastic-deterministic viewpoint for analyzing the mini-batch SGD on 
		$\psi(\boldsymbol{x})=F(\boldsymbol{x})+h(\boldsymbol{x})$. 
		The key idea is to distinguish the finite-sum $F(\boldsymbol{x})$ and the independent term $h(\boldsymbol{x})$ by the way their gradients are accessed, where $\nabla F(\boldsymbol{x})$ is sampled stochastically while $\nabla h(\boldsymbol{x})$ is computed exactly. 
		This viewpoint allows us to answer a question that is less apparent in standard finite-sum formulation: how do stochastic and deterministic gradient computations influence the mini-batch SGD differently? 
		In this sense, we provide a new perspective for understanding SGD on structured minimization problems.
	
	\item Based on this viewpoint, we establish linear convergence of the mini-batch SGD to a neighborhood of the minimizer, where the Lipschitz constants for the gradients of $f_i(\boldsymbol{x})$, $i=1,\ldots,n$, and $h(\boldsymbol{x})$ appear asymmetrically in the expressions for the step size, convergence rate, and radius of the convergence region. Based on this observation, we prove that reducing the Lipschitz constant for the gradients of $f_i(\boldsymbol{x})$, $i=1,\ldots,n$ yields a greater improvement in convergence rate compared to an equivalent reduction in the Lipschitz constant for the gradient of $h(\boldsymbol{x})$.
	
	\item Our distinct gradient analysis method offers sharper theoretical guarantees, and provides insights for model formulation and algorithm design. Specifically, we theoretically prove that for the same algorithm sequence, our analysis yields a faster convergence rate and a smaller radius of the convergence region compared to existing results, such as Theorem 5.8 from \cite{garrigos2023handbook}. Moreover, when $h(\boldsymbol{x})$ endows the objective function with sufficient strong convexity, we establish that the algorithm can achieve an even faster linear convergence rate by utilizing the specific structure of the problem.
	
	\item Our analysis also quantifies the role of batch size. 
	A smaller batch size leads to slower convergence and larger oscillations in the limiting neighborhood of the minimizer. 
	As the batch size increases, our convergence result approaches the classic linear convergence rate of full gradient descent in expectation. 
	
	\item Numerical experiments on logistic regression support these findings and further illustrate the distinct roles of $F(\boldsymbol{x})$ and $h(\boldsymbol{x})$.
\end{itemize}

The remainder of this paper is organized as follows. In \Cref{sec:2}, notations, some results from convex analysis, and the applications of the problem \eqref{eq:opt1} are provided. In \Cref{sec:3}, we present the convergence analysis of the mini-batch SGD and investigate how the stochastic approximation influences the algorithm.
In \Cref{sec:4}, we further discuss the asymmetric roles of the stochastic and deterministic gradients. Numerical results are reported in \Cref{sec:5}. Finally, conclusions are given in \Cref{sec:6}.

\section{Preliminaries}\label{sec:2}

This section introduces the notations and preliminary results used in our analysis. We then present several applications of the problem \eqref{eq:opt1}.

\subsection{Notations}

The set of $d$-dimensional real vectors is denoted by $\mathbb{R}^d$. The symbol $\boldsymbol{0}$ denotes the zero column vector with its size varying with the context. Let $\langle\cdot,\cdot\rangle$ and $\|\cdot\|_2$ denote the standard Euclidean inner product and its induced norm, respectively. The absolute value is denoted by $|\cdot|$. The probability measure is denoted by $\mathbb{P}$. Let $\mathbb{E}[\cdot]$ denote the expectation, and $\mathbb{E}[\cdot|\mathcal{F}]$ denote the conditional expectation with respect to the $\sigma$-field $\mathcal{F}$. 

\subsection{Smoothness and Strong Convexity}
\label{sec:2.2}
Let $\psi:\mathbb{R}^d \to \mathbb{R}$ be a continuously differentiable function.

\begin{definition}[Lipschitz continuous gradient \cite{nesterov2018lectures}]\label{def:Lip_con}
	A function $\psi$ has Lipschitz continuous gradient with constant $L>0$ or is called $L$-smooth if
	\[
	\|\nabla\psi(\boldsymbol{x})-\nabla\psi(\boldsymbol{y})\|_2 \le L\|\boldsymbol{x}-\boldsymbol{y}\|_2,
	\quad\text{for all }\boldsymbol{x},\boldsymbol{y}\in\mathbb{R}^d.
	\]
\end{definition}

\begin{lemma}[Lemma 1.2.3 \cite{nesterov2018lectures}]
	\label{lem:L-smooth1}
	Given that $\psi$ is $L$-smooth, the inequality
	\begin{equation}
		\label{eq:L-smooth1}
		\bigl|\psi(\boldsymbol{y})-\psi(\boldsymbol{x})-\langle\nabla\psi(\boldsymbol{x}),\boldsymbol{y}-\boldsymbol{x}\rangle\bigr|
		\le \frac{L}{2}\|\boldsymbol{y}-\boldsymbol{x}\|_2^{2}
	\end{equation}
	holds for all $\boldsymbol{x},\boldsymbol{y}\in\mathbb{R}^d$.
\end{lemma}

\begin{definition}[Strong convexity \cite{nesterov2018lectures}]
	A function $\psi$ is $\lambda$-strongly convex if the function
	\[
	\psi(\boldsymbol{x})-\frac{\lambda}{2}\|\boldsymbol{x}\|_2^{2}
	\]
	is convex for some $\lambda>0$.  The constant $\lambda$ is called the convexity parameter.
\end{definition}

\begin{lemma}[Theorem 2.1.10 \cite{nesterov2018lectures}]
	\label{lem:str-cvx}
	The inequalities 
	\begin{align}
		&\psi(\boldsymbol{y}) \geq \psi(\boldsymbol{x})+\langle\nabla \psi(\boldsymbol{x}), \boldsymbol{y}-\boldsymbol{x}\rangle+\frac{\lambda}{2} \|\boldsymbol{y}-\boldsymbol{x}\|_2^2, \label{eq:str-cvx1} \\
		&\lambda\|\boldsymbol{x}-\boldsymbol{y}\|_2 \leq\|\nabla \psi(\boldsymbol{x})-\nabla \psi(\boldsymbol{y})\|_2, \label{eq:str-cvx7}
	\end{align}
	hold for all $\boldsymbol{x},\boldsymbol{y}\in\mathbb{R}^d$ when $\psi$ is $\lambda$-strongly convex.
\end{lemma}

Combining \Cref{def:Lip_con} and inequality \eqref{eq:str-cvx7} implies that
$\lambda\leq L$ whenever $\psi$ is both $L$-smooth and $\lambda$-strongly convex. The following inequality holds for such functions as well.

\begin{lemma}[Theorem 2.1.12 \cite{nesterov2018lectures}]
	\label{lem:L-smooth-str-cvx}
	Given that $\psi$ is $L$-smooth and
	$\lambda$-strongly convex, we have
	\begin{equation}
		\langle\nabla \psi(\boldsymbol{x})-\nabla \psi(\boldsymbol{y}), \boldsymbol{x}-\boldsymbol{y}\rangle 
		\geq \frac{\lambda L}{\lambda+L}\|\boldsymbol{x}-\boldsymbol{y}\|_2^2+\frac{1}{\lambda+L}\|\nabla \psi(\boldsymbol{x})-\nabla \psi(\boldsymbol{y})\|_2^2 . \label{eq:L-smooth-str-cvx}
	\end{equation}
\end{lemma}

\subsection{Applications of Problem \eqref{eq:opt1}}
\label{sec:2.3}
This subsection starts with three classic supervised learning problems.  
For each one, the objective function $\psi(\boldsymbol{x})$ is strongly convex, with  $\nabla f_i(\boldsymbol{x})$, $i=1,\ldots,n$, and $\nabla h(\boldsymbol{x})$ being Lipschitz continuous. We then show that \eqref{eq:opt1} remains important in a wider context through a semi-supervised learning problem.

Let $\{(\boldsymbol{a}_i,b_i)\}_{i=1}^n$ be $n$ sample pairs with $\boldsymbol{a}_i\in\mathbb{R}^d$ and $b_i\in\mathbb{R}$.  
Specifying particular forms for $f_i(\boldsymbol{x})$ and $h(\boldsymbol{x})$ in \eqref{eq:opt1} yields several classic supervised learning problems \cite{scholkopf2002learning}.  

\textbf{Ridge regression.}  
For regression problems, ridge regression alleviates multicollinearity in the feature matrix.  
In this setting, 

\[
f_i(\boldsymbol{x})=\frac{1}{2}(\boldsymbol{a}_i^{\top} \boldsymbol{x} - b_i)^2, i=1,\ldots,n,\; h(\boldsymbol{x})=\frac{L_h}{2}\|\boldsymbol{x}\|_2^2,
\]  
where $L_h$ is a hyperparameter that balances the two terms.  
This model is effective when the matrix formed by $\{\boldsymbol{a}_i\}_{i=1}^{n}$ has a large condition number, such as in adaptive beamforming \cite{selen2008automatic}.

\textbf{Logistic regression.}  
For binary classification, we may use logistic regression.  
Here $b_i\in\{1,-1\}$ and 

\[
f_i(\boldsymbol{x})=\log\bigl(1+\exp(-b_i \boldsymbol{a}_i^{\top} \boldsymbol{x})\bigr), i=1,\ldots,n, 
h(\boldsymbol{x})=\frac{L_h}{2}\|\boldsymbol{x}\|_2^2.
\]  
This model is applied in click-through-rate estimation~\cite{mcmahan2013ad} and other scenarios that require a clear probabilistic interpretation.  

\textbf{Support vector machine.}  
When a classification task requires a large margin or robustness to outliers, the support vector machine (SVM) is a common choice.  
For SVM with the squared hinge loss, the problem \eqref{eq:opt1} takes the form

\[
f_i(\boldsymbol{x})=\frac{1}{2}\max\{0,1-b_i \boldsymbol{a}_i^{\top} \boldsymbol{x}\}^2, i=1,\ldots,n,
h(\boldsymbol{x})=\frac{L_h}{2}\|\boldsymbol{x}\|_2^2.
\]  
Such a setting suits applications that need clear decision boundaries, such as spam filtering \cite{drucker1999support}.  

In the above three problems, $\nabla f_i(\boldsymbol{x})$, $i=1,\ldots,n$, and $\nabla h(\boldsymbol{x})$ are Lipschitz continuous.  
The Lipschitz constants for $\nabla f_i(\boldsymbol{x})$ in each problem are $\|\boldsymbol{a}_i\|_2^2$, $\|\boldsymbol{a}_i\|_2^2/4$, and $\|\boldsymbol{a}_i\|_2^2$, respectively, while the Lipschitz constant for $\nabla h(\boldsymbol{x})$ is $L_h$.  
The objective function $\psi(\boldsymbol{x})$ is strongly convex with convexity parameter $\lambda=L_h$. In practice, these models often deal with datasets containing millions of samples or more \cite{fan2008liblinear}.  
Computing the full gradient $\nabla F(\boldsymbol{x})=\tfrac{1}{n}\sum_{i=1}^{n}\nabla f_i(\boldsymbol{x})$ is then expensive.  
A natural remedy is to estimate $\nabla F(\boldsymbol{x})$ stochastically, whereas $\nabla h(\boldsymbol{x})$ can still be obtained deterministically since its computation is independent of the data size.  
This is precisely the setting that we analyze in depth. 

\textbf{Semi-supervised learning problem.} Recent strides in machine learning lead to more complicated $F(\boldsymbol{x})$ and $h(\boldsymbol{x})$. The semi-supervised learning problem is an example, which uses limited labeled data and plentiful unlabeled data \cite{zhu2009introduction}. For instance, satellite image analysis deals with abundant images but acquiring labels (e.g., poverty levels) is expensive \cite{jean2018semi}. Semi-supervised methods train models on a small labeled set together with a large unlabeled set, and generally outperform those using only labeled data \cite{zhu2009introduction}. In some cases, the objective functions take the form $F(\boldsymbol{x})+h(\boldsymbol{x})$. When $F(\boldsymbol{x})$ uses a large unlabeled set with the finite-sum structure and $h(\boldsymbol{x})$ uses limited labeled data or lacks such a structure, solving the problem \eqref{eq:opt1} with the mini-batch SGD fits the scenario in this study. Although the objective function may not be globally strongly convex, it may be strongly convex near local minimizers \cite{bottou2018optimization}.

A concrete example is the semi-supervised deep kernel learning (SSDKL) proposed by Jean et al. \cite{jean2018semi}. The model first maps the input through a neural network to a feature space and then introduces a parametric Gaussian process kernel in that space. All trainable parameters form the vector $\boldsymbol{x}$ in the problem \eqref{eq:opt1}. Let the unlabeled set be $\mathcal{J}_U=\{\boldsymbol{a}_i\}_{i=1}^{n}$ and the labeled set be $\mathcal{J}_L=\{(\boldsymbol{a}_j,b_j)\}_{j=n+1}^{n+m}$ with $m\ll n$. The objective function of SSDKL is

\begin{equation}\label{eq:SSDKL_obj}
	\psi(\boldsymbol{x})= \frac{\alpha}{n} \sum_{\boldsymbol{a}_i \in \mathcal{J}_U} v(\boldsymbol{a}_i,\mathcal{J}_L;\boldsymbol{x}) + \frac{1}{m} \ell(\mathcal{J}_L;\boldsymbol{x}),
\end{equation}
where $v(\boldsymbol{a}_i,\mathcal{J}_L;\boldsymbol{x})$ is the predictive variance at $\boldsymbol{a}_i$ and $\ell(\mathcal{J}_L;\boldsymbol{x})$ is the negative log-likelihood on the labeled set. The constant $\alpha$ balances the two terms. Setting $f_i(\boldsymbol{x})=\alpha v(\boldsymbol{a}_i,\mathcal{J}_L;\boldsymbol{x})$ for $i=1,\dots,n$ and $h(\boldsymbol{x})=\tfrac{1}{m} \ell(\mathcal{J}_L;\boldsymbol{x})$ rewrites \eqref{eq:SSDKL_obj} into the form of the problem \eqref{eq:opt1}. The part $F(\boldsymbol{x})=\frac{\alpha}{n}\sum_{\boldsymbol{a}_i\in \mathcal{J}_U} v(\boldsymbol{a}_i,\mathcal{J}_L;\boldsymbol{x})$
has the finite-sum structure and involves the large unlabeled set, so its gradient can be estimated by stochastic approximation. The term $h(\boldsymbol{x})=\frac{1}{m} \ell(\mathcal{J}_L;\boldsymbol{x})$
lacks a finite-sum form, but computing its gradient directly is relatively cheap due to the small size of $\mathcal{J}_L$. Hence, applying the mini-batch SGD to \eqref{eq:SSDKL_obj} fits the scenario we consider, and it is indeed the optimizer chosen in \cite{jean2018semi}.

\section{Convergence Analysis}\label{sec:3}

In this section, we present and discuss the main theorem of this paper. For convenience, we refer to $F(\boldsymbol{x})$ as the SG-term (indicating that its gradient is obtained via stochastic approximation) and to $h(\boldsymbol{x})$ as the DG-term (reflecting its deterministic gradient calculation).

An assumption for the objective function $\psi(\boldsymbol{x})=F(\boldsymbol{x})+h(\boldsymbol{x})$ in \eqref{eq:opt1}, with $F(\boldsymbol{x}):=\tfrac{1}{n}\sum_{i=1}^{n} f_i(\boldsymbol{x})$, is provided as follows.

\noindent \textbf{Assumption} Each $f_i(\boldsymbol{x})$ is $L_i$-smooth for $i=1,\ldots,n$, the function $h(\boldsymbol{x})$ is $L_h$-smooth, and $\psi(\boldsymbol{x})$ is $\lambda$-strongly convex.

A similar assumption can be found in~\cite{nguyen2018sgd}.

Let $\mathcal{S}_k=\{\xi_{k,1},\ldots,\xi_{k,B}\}$ be a set of $B$ independent and identically distributed random variables. Each $\xi_{k,i}$ takes values in $\{1,\ldots,n\}$ and satisfies $\mathbb{P}(\xi_{k,i}=j)=p_j>0$ for $i=1,\ldots,B$ and $j=1,\ldots,n$. The formal framework of the mini-batch SGD is presented below.

\begin{algorithm}[H]
	\caption*{\textbf{Algorithm} Mini-batch SGD}
	\begin{algorithmic}[1]
		\REQUIRE Step size $\eta$, initial point $\boldsymbol{x}_1$.
		\FOR{$k=1,\ldots,K$}
		\STATE Compute the mini-batch stochastic gradient $\boldsymbol{g}_k=\frac{1}{B}\sum_{j=1}^{B} \frac{1}{\textstyle np_{\scriptscriptstyle \xi_{k,j}}}\nabla f_{\xi_{k,j}}(\boldsymbol{x}_k)$.
		\STATE $\boldsymbol{x}_{k+1}=\boldsymbol{x}_k-\eta(\boldsymbol{g}_k+\nabla h(\boldsymbol{x}_k))$.
		\ENDFOR
		\ENSURE $\boldsymbol{x}_{K+1}$.
	\end{algorithmic}
\end{algorithm}

Since $\xi_{k,1},\ldots,\xi_{k,B}$ are independent and identically distributed, the mini-batch gradient is an unbiased estimator of the full gradient:
\begin{small}
	\begin{equation*}
		\mathbb{E}\left[\frac{1}{B}\sum\limits_{j=1}^B \frac{1}{n\textstyle p_{\scriptscriptstyle \xi_{k,j}}}\nabla f_{\xi_{k,j}}(\boldsymbol{x}_k)\middle| \mathcal{F}_k\right]=\frac{1}{B}\sum\limits_{j=1}^B \mathbb{E}\left[\frac{1}{n\textstyle p_{\scriptscriptstyle \xi_{k,j}}}\nabla f_{\xi_{k,j}}(\boldsymbol{x}_k)\middle| \mathcal{F}_k\right] =\frac{1}{n}\sum_{i=1}^n\nabla f_i(\boldsymbol{x}_k),
	\end{equation*}
\end{small}
where $\mathcal{F}_k$ is the $\sigma$-field generated by $\boldsymbol{x}_1,\ldots,\boldsymbol{x}_k$, $k=1,\ldots,K$.

The following is our main theorem, which shows that the Lipschitz constants $L_i$ for the gradients of $f_i(\boldsymbol{x})$, $i=1\ldots,n$, in the SG-term $F(\boldsymbol{x})$ and the Lipschitz constant $L_h$ for the gradient of the DG-term $h(\boldsymbol{x})$ affect the algorithm in distinct ways when applying the mini-batch SGD to \eqref{eq:opt1}.

\begin{theorem}\label{thm:1}
	Let $\boldsymbol{x}^*$ be the minimizer of the  problem \eqref{eq:opt1} and $\{\boldsymbol{x}_k\}_{k=1}^{K+1}$ be the sequence generated by the mini-batch SGD with the step size $\eta$ satisfying
	\begin{equation}\label{eq:eta4}
		0<\eta \leq \frac{2}{\bigl(s_{\scriptscriptstyle F} +1\bigr)\Big(\frac{1}{n}\sum\limits_{i=1}^{n} L_i + L_h+\lambda\Big)}.	
	\end{equation}
	Then, the following inequality
	\begin{equation}\label{eq:linear_convergence4}
		\mathbb{E}\bigl[\|\boldsymbol{x}_{K+1}-\boldsymbol{x}^*\|_2^2\bigr] \leq (1-q)^K\|\boldsymbol{x}_1-\boldsymbol{x}^*\|_2^2+R,	
	\end{equation}
	holds, where
	\begin{equation}\label{eq:q4}
		q:= \frac{2\eta\lambda\Big(\tfrac{1}{n}\sum\limits_{i=1}^{n} L_i+L_h\Big)}{\frac{1}{n}\sum\limits_{i=1}^{n} L_i + L_h+\lambda} \in(0,1),\quad
		R:= \frac{2\eta^2\sigma_{\scriptscriptstyle F}}{q},	
	\end{equation}
	\begin{equation}\label{eq:s_F}
		s_{\scriptscriptstyle F}:=\frac{2}{B\lambda^2 n}\sum_{i=1}^{n} \frac{L_i^2}{n p_i},
	\end{equation}
	\begin{equation}\label{eq:sigma_F}
		\sigma_{\scriptscriptstyle F}:=\frac{1}{B}\Bigg(\frac{1}{ n}\sum_{i=1}^{n}\frac{1}{n p_i}\|\nabla f_i\bigl(\boldsymbol{x}^*\bigr)\|_2^2-\|\nabla F\bigl(\boldsymbol{x}^*\bigr)\|_2^2\Bigg).
	\end{equation}
\end{theorem}

\begin{proof}
	Let us start at $k=K$. From the iterative scheme of mini-batch SGD, we have	
	\begin{align}
		& \mathbb{E}\left[\left\|\boldsymbol{x}_{k+1}-\boldsymbol{x}^*\right\|_2^2 \middle| \mathcal{F}_k \right] \nonumber\\
		=&\mathbb{E}\Bigg[\Bigg\|\boldsymbol{x}_k-\boldsymbol{x}^*-\eta \Bigg(\frac{1}{B}\Bigg(\sum_{j=1}^{B}\frac{1}{n\textstyle p_{\scriptscriptstyle \xi_{k,j}}}\nabla f_{\xi_{k,j}}(\boldsymbol{x}_k)\Bigg)+\nabla h(\boldsymbol{x}_k)\Bigg)\Bigg\|_{2}^{2} \Bigg| \mathcal{F}_k\Bigg] \nonumber\\
		=&\left\|\boldsymbol{x}_k-\boldsymbol{x}^*\right\|_2^2 -2 \eta \mathbb{E}\Bigg[\Bigg\langle \frac{1}{B}\Bigg(\sum_{j=1}^{B} \frac{1}{n\textstyle p_{\scriptscriptstyle \xi_{k,j}}} \nabla f_{\xi_{k,j}}(\boldsymbol{x}_k)\Bigg)+\nabla h(\boldsymbol{x}_k), \boldsymbol{x}_k-\boldsymbol{x}^*\Bigg\rangle \Bigg| \mathcal{F}_k \Bigg]\nonumber \\
		&+\eta^2 \mathbb{E}\Bigg[\Bigg\|\frac{1}{B}\Bigg(\sum_{j=1}^{B}\frac{1}{n\textstyle p_{\scriptscriptstyle \xi_{k,j}}}\nabla f_{\xi_{k,j}}(\boldsymbol{x}_k)\Bigg)+\nabla h(\boldsymbol{x}_k) \Bigg\|_{2}^{2} \Bigg| \mathcal{F}_k \Bigg] \nonumber\\
		=&\left\|\boldsymbol{x}_k-\boldsymbol{x}^*\right\|_2^2-2 \eta\left\langle\nabla F(\boldsymbol{x}_k)+\nabla h(\boldsymbol{x}_k), \boldsymbol{x}_k-\boldsymbol{x}^*\right\rangle \nonumber \\
		&+\eta^2 \mathbb{E}\Bigg[\Bigg\|\frac{1}{B}\Bigg(\sum_{j=1}^{B}\frac{1}{n\textstyle p_{\scriptscriptstyle \xi_{k,j}}}\nabla f_{\xi_{k,j}}(\boldsymbol{x}_k)\Bigg)+\nabla h(\boldsymbol{x}_k) \Bigg\|_{2}^{2} \Bigg| \mathcal{F}_k\Bigg].\label{eq:total_ineq2_mini}
	\end{align}
	
	Since $\psi(\boldsymbol{x})=F(\boldsymbol{x})+h(\boldsymbol{x})=\frac{1}{n}\sum_{i=1}^{n}f_i(\boldsymbol{x})+h(\boldsymbol{x})$ is $(\tfrac{1}{n}\sum_{i=1}^{n} L_i+L_h)$-smooth and $\lambda$-strongly convex, from \eqref{eq:L-smooth-str-cvx}, we have
	\begin{align}
		&\left\langle\nabla F(\boldsymbol{x}_k)+\nabla h(\boldsymbol{x}_k), \boldsymbol{x}_k-\boldsymbol{x}^*\right\rangle \nonumber \\
		\geq &\frac{\lambda\Big(\frac{1}{n}\sum\limits_{i=1}^{n}L_i+L_h\Big)}{\frac{1}{n}\sum\limits_{i=1}^{n}L_i+L_h+\lambda}\|\boldsymbol{x}_k-\boldsymbol{x}^*\|_2^2
		+\frac{1}{\frac{1}{n}\sum\limits_{i=1}^{n}L_i+L_h+\lambda}\|\nabla F(\boldsymbol{x}_k)+\nabla h(\boldsymbol{x}_k)\|_2^2.\label{eq:Lsmooth_strcvx}
	\end{align}

	Combining \eqref{eq:total_ineq2_mini} with \eqref{eq:Lsmooth_strcvx} yields
	
	\begin{align}
		\mathbb{E}\left[\left\|\boldsymbol{x}_{k+1}-\boldsymbol{x}^*\right\|_2^2 \Big| \mathcal{F}_k\right] 
		\leq& \Bigg(1-\frac{2\eta\lambda(\frac{1}{n}\sum\limits_{i=1}^{n}L_i+L_h)}{\frac{1}{n}\sum\limits_{i=1}^{n}L_i+L_h+\lambda}\Bigg)\left\|\boldsymbol{x}_k-\boldsymbol{x}^*\right\|_2^2 \nonumber \\
		&-\frac{2\eta}{\frac{1}{n}\sum\limits_{i=1}^{n}L_i+L_h+\lambda}\|\nabla F(\boldsymbol{x}_k)+\nabla h(\boldsymbol{x}_k)\|_2^2\nonumber \\
		&+\eta^2 \mathbb{E}\Bigg[\Bigg\|\frac{1}{B}\Bigg(\sum_{j=1}^{B}\frac{1}{n\textstyle p_{\scriptscriptstyle \xi_{k,j}}}\nabla f_{\xi_{k,j}}(\boldsymbol{x}_k)\Bigg)+\nabla h(\boldsymbol{x}_k) \Bigg\|_2^2 \Bigg| \mathcal{F}_k\Bigg].\label{eq:part1_mini}
	\end{align}
	
	For the last term in \eqref{eq:part1_mini}, notice that $\mathbb{E}\Big[\frac{1}{B}\Big(\sum_{j=1}^{B}\frac{1}{\textstyle n p_{\scriptscriptstyle \xi_{k,j}}}$ $\nabla f_{\xi_{k,j}}(\boldsymbol{x}_k)\Big) \Big| \mathcal{F}_k\Big]=\nabla F(\boldsymbol{x}_k)$. As a result, we have
	\begin{align}
		& \mathbb{E}\Bigg[\Bigg\|\frac{1}{B}\Bigg(\sum_{j=1}^{B}\frac{1}{n\textstyle p_{\scriptscriptstyle \xi_{k,j}}}\nabla f_{\xi_{k,j}}(\boldsymbol{x}_k)\Bigg)+\nabla h\left(\boldsymbol{x}_k\right)\Bigg\|_2^2 \Bigg| \mathcal{F}_k\Bigg] \nonumber\\
		= & \mathbb{E}\Bigg[\Bigg\|\frac{1}{B}\left(\sum_{j=1}^{B}\frac{1}{n\textstyle p_{\scriptscriptstyle \xi_{k,j}}}\nabla f_{\xi_{k,j}}(\boldsymbol{x}_k)\right)-\nabla F\left(\boldsymbol{x}_k\right)  +\nabla F\left(\boldsymbol{x}_k\right)+\nabla h\left(\boldsymbol{x}_k\right) \Bigg\|_2^2 \Bigg| \mathcal{F}_k\Bigg]\nonumber \\
		= & \mathbb{E}\Bigg[\Bigg\|\frac{1}{B}\Bigg(\sum_{j=1}^{B}\frac{1}{n\textstyle p_{\scriptscriptstyle \xi_{k,j}}}\nabla f_{\xi_{k,j}}(\boldsymbol{x}_k)\Bigg)-\nabla F(\boldsymbol{x}_k)\Bigg\|_2^2 \Bigg| \mathcal{F}_k\Bigg] + \left\|\nabla F\left(\boldsymbol{x}_k\right)+\nabla h\left(\boldsymbol{x}_k\right)\right\|_2^2\nonumber\\
		& + 2\mathbb{E}\Bigg[\Bigg\langle \frac{1}{B}\Bigg(\sum_{j=1}^{B}\frac{1}{n\textstyle p_{\scriptscriptstyle \xi_{k,j}}}\nabla f_{\xi_{k,j}}(\boldsymbol{x}_k)\Bigg)-\nabla F\left(\boldsymbol{x}_k\right), 
		\nabla F\left(\boldsymbol{x}_k\right)+\nabla h\left(\boldsymbol{x}_k\right) \Bigg\rangle \Bigg| \mathcal{F}_k\Bigg] \nonumber \\
		= & \mathbb{E}\Bigg[\Bigg\|\frac{1}{B}\Bigg(\sum_{j=1}^{B}\frac{1}{n\textstyle p_{\scriptscriptstyle \xi_{k,j}}}\nabla f_{\xi_{k,j}}(\boldsymbol{x}_k)\Bigg)-\nabla F\left(\boldsymbol{x}_k\right)\Bigg\|_2^2\Bigg|\mathcal{F}_k\Bigg]\nonumber \\
		&+\left\|\nabla F\left(\boldsymbol{x}_k\right)+\nabla h\left(\boldsymbol{x}_k\right)\right\|_2^2.
		\label{eq:part2_mini}
	\end{align}
	
	For the first term in \eqref{eq:part2_mini}, notice that $\mathbb{E}\Big[\frac{1}{\textstyle n p_{\scriptscriptstyle \xi_{k,j}}}\nabla f_{\xi_{k,j}}\left(\boldsymbol{x}_k\right)$ $ \Big| \mathcal{F}_k\Big]=\nabla F(\boldsymbol{x}_k)$ and the random variables $\xi_{k,j},j=1,\ldots,B$, are independent and identically distributed. So we have
	\begin{align}
		&\mathbb{E}\Bigg[\Bigg\|\frac{1}{B}\Bigg(\sum\limits_{j=1}^{B}\frac{1}{\textstyle n p_{\scriptscriptstyle \xi_{k,j}}}\nabla f_{\xi_{k,j}}(\boldsymbol{x}_k)\Bigg)-\nabla F\left(\boldsymbol{x}_k\right)\Bigg\|_2^2 \Bigg| \mathcal{F}_k\Bigg] \nonumber \\
		=& \frac{1}{B^2}\mathbb{E}\Bigg[\Bigg\|\sum\limits_{j=1}^{B}\Bigg(\frac{1}{\textstyle n p_{\scriptscriptstyle \xi_{k,j}}}\nabla f_{\xi_{k,j}}(\boldsymbol{x}_k)-\nabla F\left(\boldsymbol{x}_k\right)\Bigg)\Bigg\|_2^2 \Bigg| \mathcal{F}_k\Bigg] \nonumber\\
		=& \frac{1}{B^2}\mathbb{E}\Bigg[\sum\limits_{j=1}^{B}\Bigg\|\frac{1}{\textstyle n p_{\scriptscriptstyle \xi_{k,j}}}\nabla f_{\xi_{k,j}}(\boldsymbol{x}_k)-\nabla F\left(\boldsymbol{x}_k\right)\Bigg\|_2^2 \nonumber \\
		&+\sum\limits_{j\neq l}\Bigg\langle\frac{1}{\textstyle n p_{\scriptscriptstyle \xi_{k,j}}}\nabla f_{\xi_{k,j}}(\boldsymbol{x}_k)-\nabla F(\boldsymbol{x}_k), \frac{1}{\textstyle n p_{\scriptscriptstyle \xi_{k,l}}}\nabla f_{\xi_{k,l}}(\boldsymbol{x}_k)-\nabla F(\boldsymbol{x}_k) \Bigg\rangle\Bigg| \mathcal{F}_k\Bigg]\nonumber \\
		=&\frac{1}{B^2}\mathbb{E}\Bigg[\sum_{j=1}^{B}\Bigg\|\frac{1}{\textstyle n p_{\scriptscriptstyle \xi_{k,j}}}\nabla f_{\xi_{k,j}}(\boldsymbol{x}_k)-\nabla F\left(\boldsymbol{x}_k\right)\Bigg\|_2^2\Bigg| \mathcal{F}_k\Bigg]\nonumber \\
		=&\frac{1}{B}\mathbb{E}\Bigg[\Bigg\|\frac{1}{\textstyle n p_{\scriptscriptstyle \xi_{k,1}}}\nabla f_{\xi_{k,1}}(\boldsymbol{x}_k)-\nabla F\left(\boldsymbol{x}_k\right)\Bigg\|_2^2\Bigg| \mathcal{F}_k\Bigg].\label{eq:part3_mini}
	\end{align}
	
	Since $f_i(\boldsymbol{x})$ is $L_i$-smooth, $i=1,\ldots,n$, the following inequality holds:
	\begin{align}
		&\mathbb{E}\Bigg[\Bigg\|\frac{1}{\textstyle n p_{\scriptscriptstyle \xi_{k,1}}}\nabla f_{\xi_{k,1}}(\boldsymbol{x}_k)-\nabla F\left(\boldsymbol{x}_k\right)\Bigg\|_2^2 \Bigg| \mathcal{F}_k\Bigg]\nonumber\\
		=&\mathbb{E}\Bigg[\Bigg\|\frac{1}{\textstyle n p_{\scriptscriptstyle \xi_{k,1}}}\nabla f_{\xi_{k,1}}(\boldsymbol{x}_k)-\nabla F\left(\boldsymbol{x}_k\right)-\Bigg(\frac{1}{\textstyle n p_{\scriptscriptstyle \xi_{k,1}}}\nabla f_{\xi_{k,1}}\left(\boldsymbol{x}^*\right)-\nabla F(\boldsymbol{x}^*)\Bigg)\nonumber \\
		&+\Bigg(\frac{1}{\textstyle n p_{\scriptscriptstyle \xi_{k,1}}}\nabla f_{\xi_{k,1}}\left(\boldsymbol{x}^*\right)-\nabla F(\boldsymbol{x}^*)\Bigg)\Bigg\|_{2}^{2}  \Bigg|  \mathcal{F}_k\Bigg] \nonumber \\
		\leq &\mathbb{E}\Bigg[2\Bigg\|\frac{1}{\textstyle n p_{\scriptscriptstyle \xi_{k,1}}}\nabla f_{\xi_{k,1}}(\boldsymbol{x}_k)-\nabla F\left(\boldsymbol{x}_k\right)-\Bigg(\frac{1}{\textstyle n p_{\scriptscriptstyle \xi_{k,1}}}\nabla f_{\xi_{k,1}}\left(\boldsymbol{x}^*\right)-\nabla F(\boldsymbol{x}^*)\Bigg)\Bigg\|_2^2 \nonumber \\
		&+2\Bigg\|\frac{1}{\textstyle n p_{\scriptscriptstyle \xi_{k,1}}}\nabla f_{\xi_{k,1}}\left(\boldsymbol{x}^*\right)-\nabla F(\boldsymbol{x}^*)\Bigg\|_2^2\Bigg|\mathcal{F}_k\Bigg]\nonumber\\
		=&2\mathbb{E}\Bigg[\Bigg\|\frac{1}{\textstyle n p_{\scriptscriptstyle \xi_{k,1}}}\nabla f_{\xi_{k,1}}(\boldsymbol{x}_k)-\frac{1}{\textstyle n p_{\scriptscriptstyle \xi_{k,1}}}\nabla f_{\xi_{k,1}}(\boldsymbol{x}^*)\Bigg\|_2^2 \nonumber \\
		&-2\Bigg\langle\frac{1}{\textstyle np_{\scriptscriptstyle \xi_{k,1}}}\nabla f_{\xi_{k,1}}(\boldsymbol{x}_k)-\frac{1}{\textstyle np_{\scriptscriptstyle \xi_{k,1}}}\nabla f_{\xi_{k,1}}(\boldsymbol{x}^*),\nabla F(\boldsymbol{x}_k)-\nabla F(\boldsymbol{x}^*)\Bigg\rangle \nonumber \\
		&+\|\nabla F(\boldsymbol{x}_k)-\nabla F(\boldsymbol{x}^*)\|_2^{2}+ \Bigg\|\frac{1}{\textstyle n p_{\scriptscriptstyle \xi_{k,1}}}\nabla f_{\xi_{k,1}}(\boldsymbol{x}^*)\Bigg\|_2^2
		\nonumber \\
		&-2\Bigg\langle\frac{1}{\textstyle n p_{\scriptscriptstyle \xi_{k,1}}}\nabla f_{\xi_{k,1}}(\boldsymbol{x}^*), \nabla F(\boldsymbol{x}^*)\Bigg\rangle+\|\nabla F(\boldsymbol{x}^*)\|_2^2\Bigg|\mathcal{F}_k\Bigg] \nonumber \\
		=&\frac{2}{n}\sum_{i=1}^n\frac{1}{np_i}\left\|\nabla f_{i}(\boldsymbol{x}_k)-\nabla f_{i}\left(\boldsymbol{x}^*\right)\right\|_2^2 - 2\|\nabla F(\boldsymbol{x}_k)-\nabla F(\boldsymbol{x}^*)\|_2^2 \nonumber \\
		&+\frac{2}{n}\sum_{i=1}^n\frac{1}{np_i}\left\|\nabla f_{i}\left(\boldsymbol{x}^*\right)\right\|_2^2-2\left\|\nabla F(\boldsymbol{x}^*)\right\|_2^2 \nonumber \\
		\leq&\frac{2}{n}\sum_{i=1}^n\frac{1}{np_i}\left\|\nabla f_{i}(\boldsymbol{x}_k)-\nabla f_{i}\left(\boldsymbol{x}^*\right)\right\|_2^2+\frac{2}{n}\sum_{i=1}^n\frac{1}{np_i}\left\|\nabla f_{i}\left(\boldsymbol{x}^*\right)\right\|_2^2-2\left\|\nabla F(\boldsymbol{x}^*)\right\|_2^2 \nonumber \\
		\leq &\frac{2}{n} \sum_{i=1}^n \frac{L_i^2}{n p_i} \|\boldsymbol{x}_k-\boldsymbol{x}^*\|_2^2 + \frac{2}{n}\sum_{i=1}^n\frac{1}{np_i}\left\|\nabla f_{i}\left(\boldsymbol{x}^*\right)\right\|_2^2-2\left\|\nabla F(\boldsymbol{x}^*)\right\|_2^2\label{eq:part3_mini_con}
	\end{align}
	
	The third equality in \eqref{eq:part3_mini_con} follows from the definition of the expectation with respect to the discrete random variable $\xi_{k,1}$. Specifically, for any vector $\boldsymbol{v}_{\xi_{k,1}}$ that depends on  $\xi_{k,1}$, the expectation of the squared norm satisfies $$\mathbb{E}\bigg[\bigg\|\frac{1}{\textstyle n p_{\scriptscriptstyle \xi_{k,1}}} \boldsymbol{v}_{\xi_{k,1}}\bigg\|_2^2\bigg]=\sum_{i=1}^n p_i \frac{1}{(n p_i)^2}\|\boldsymbol{v}_i\|_2^2=\frac{1}{n}\sum_{i=1}^n \frac{1}{\textstyle n p_{i}} \|\boldsymbol{v}_{i}\|_2^2. $$
	
	Since $F(\boldsymbol{x})+h(\boldsymbol{x})$ is $\lambda$-strongly convex, the inequality \eqref{eq:str-cvx7} gives an upper bound on $\|\boldsymbol{x}_k - \boldsymbol{x}^*\|_2^2$ by taking $\boldsymbol{x} = \boldsymbol{x}_k$ and $\boldsymbol{y} = \boldsymbol{x}^*$, and using $\nabla F(\boldsymbol{x}^*) + \nabla h(\boldsymbol{x}^*)=\boldsymbol{0}$. Replacing the term $\|\boldsymbol{x}_k - \boldsymbol{x}^*\|_2^2$ in \eqref{eq:part3_mini_con} by this bound yields
	\begin{equation}
		\frac{1}{B}\mathbb{E}\bigg[\bigg\|\frac{1}{\textstyle n p_{\scriptscriptstyle\xi_{k,1}}}\nabla f_{\xi_{k,1}}(\boldsymbol{x}_k)-\nabla F(\boldsymbol{x}_k)\bigg\|_{2}^{2}\bigg|\mathcal{F}_k\bigg]
		\leq s_{\scriptscriptstyle F}\|\nabla F\left(\boldsymbol{x}_k\right)+\nabla h\left(\boldsymbol{x}_k\right)\|_{2}^{2}+2\sigma_{\scriptscriptstyle F},\label{eq:part4_mini}
	\end{equation}
	where $s_{\scriptscriptstyle F}=\frac{2}{B\lambda^2n}\sum_{i=1}^{n} \frac{L_i^2}{np_i}$ and $\sigma_{\scriptscriptstyle F}=\tfrac{1}{B}\left(\frac{1}{n}\sum_{i=1}^n\frac{1}{np_i}\left\|\nabla f_i\left(\boldsymbol{x}^*\right)\right\|_2^2-\left\|\nabla F(\boldsymbol{x}^*)\right\|_2^2\right)$.
	
	Combining \eqref{eq:part1_mini}, \eqref{eq:part2_mini}, \eqref{eq:part3_mini}, and \eqref{eq:part4_mini}, we obtain
	\begin{align}
		&\mathbb{E}\left[\left\|\boldsymbol{x}_{k+1}-\boldsymbol{x}^*\right\|_2^2 \Big| \mathcal{F}_k\right] \nonumber \\ 
		\leq& \Bigg(1-\frac{2\eta\lambda\Big(\frac{1}{n}\sum\limits_{i=1}^{n} L_i+L_h\Big)}{\frac{1}{n}\sum\limits_{i=1}^{n} L_i + L_h+\lambda}\Bigg)\left\|\boldsymbol{x}_k-\boldsymbol{x}^*\right\|_2^2 \nonumber \\
		&- \Bigg(\frac{2\eta}{\frac{1}{n}\sum\limits_{i=1}^{n} L_i + L_h+\lambda}-\left(s_{\scriptscriptstyle F}+1\right)\eta^2\Bigg)\|\nabla F(\boldsymbol{x}_k)+\nabla h(\boldsymbol{x}_k)\|_{2}^{2} + 2\eta^2\sigma_{\scriptscriptstyle F}. \label{eq:final_part1_mini}
	\end{align}
	
	Let $q= (2\eta\lambda(\frac{1}{n}\sum_{i=1}^{n} L_i+L_h))/(\frac{1}{n}\sum_{i=1}^{n} L_i+L_h+\lambda)$. If $0<\eta \leq 2/((s_{\scriptscriptstyle F} +1)(\lambda + \frac{1}{n}\sum_{i=1}^{n} L_i +L_h))$, then we have $q>0$, and
	\begin{align*}
		1 - q
		\geq& 1 - \frac{4\lambda\Big(\frac{1}{n}\sum\limits_{i=1}^n L_i + L_h\Big)}
		{\left(s_{\scriptscriptstyle F}+1\right)\Big(\frac{1}{n}\sum\limits_{i=1}^{n} L_i + L_h+\lambda\Big)^2}\\
		=&\frac{\lambda^2 - 2\lambda\Big(\frac{1}{n}\sum\limits_{i=1}^n L_i+L_h\Big)
			+\Big(\frac{1}{n}\sum\limits_{i=1}^n L_i+L_h\Big)^2
			+\lambda^2 s_{\scriptscriptstyle F}
		}
		{\left(s_{\scriptscriptstyle F}+1\right)\Big(\frac{1}{n}\sum\limits_{i=1}^{n} L_i + L_h+\lambda\Big)^2}\\
		&+ \frac{2\lambda s_{\scriptscriptstyle F}\Big(\frac{1}{n}\sum\limits_{i=1}^n L_i+L_h\Big)
			+ s_{\scriptscriptstyle F}\Big(\frac{1}{n}\sum\limits_{i=1}^n L_i+L_h\Big)^2}{\left(s_{\scriptscriptstyle F}+1\right)\Big(\frac{1}{n}\sum\limits_{i=1}^{n} L_i + L_h+\lambda\Big)^2}\\
		>&\frac{\lambda^2 - 2\lambda\Big(\frac{1}{n}\sum\limits_{i=1}^n L_i+L_h\Big)
			+\Big(\frac{1}{n}\sum\limits_{i=1}^n L_i+L_h\Big)^2
			+\lambda^2 s_{\scriptscriptstyle F}
		}
		{\left(s_{\scriptscriptstyle F}+1\right)\Big(\frac{1}{n}\sum\limits_{i=1}^{n} L_i + L_h+\lambda\Big)^2}\\
		&- \frac{2\lambda s_{\scriptscriptstyle F}\Big(\frac{1}{n}\sum\limits_{i=1}^n L_i+L_h\Big)
			+ s_{\scriptscriptstyle F}\Big(\frac{1}{n}\sum\limits_{i=1}^n L_i+L_h\Big)^2}{\left(s_{\scriptscriptstyle F}+1\right)\Big(\frac{1}{n}\sum\limits_{i=1}^{n} L_i + L_h+\lambda\Big)^2}\\
		=&\Bigg(\frac{\frac{1}{n}\sum\limits_{i=1}^n L_i+L_h-\lambda}
		{\frac{1}{n}\sum\limits_{i=1}^n L_i+L_h+\lambda}\Bigg)^2\geq 0.
	\end{align*}
	
	Since $\frac{2\eta}{\frac{1}{n}\sum_{i=1}^{n} L_i + L_h+\lambda}-\left(s_{\scriptscriptstyle F}+1\right)\eta^2$ $\geq 0$, \eqref{eq:final_part1_mini} reduces to
	\begin{equation}\label{eq:final_part2_mini}
		\mathbb{E}\left[\left\|\boldsymbol{x}_{k+1}-\boldsymbol{x}^*\right\|_2^2 \middle| \mathcal{F}_k\right] \leq \left(1-q\right)\left\|\boldsymbol{x}_k-\boldsymbol{x}^*\right\|_2^2 + 2\eta^2 \sigma_{\scriptscriptstyle F}.
	\end{equation}
	
	Finally, we have
	
	\begin{align}
		\mathbb{E}\left[\left\|\boldsymbol{x}_{k+1}-\boldsymbol{x}^*\right\|_2^2 -\frac{2\eta^2 \sigma_{\scriptscriptstyle F}}{q}\middle| \mathcal{F}_k\right]\leq \left(1-q\right)\left(\left\|\boldsymbol{x}_k-\boldsymbol{x}^*\right\|_2^2-\frac{2 \eta^2 \sigma_{\scriptscriptstyle F}}{q}\right).
	\end{align}
	
	Extending the same procedure to $k = K-1$ and subsequently to others completes the proof.
\end{proof}

\Cref{thm:1} shows that the algorithm converges linearly to a neighborhood of the minimizer $\boldsymbol{x}^*$ with the radius $R$. To study how the stochastic approximation of the gradient of the SG-term affects the algorithm's performance, we compare \Cref{thm:1} with the classic convergence result of the deterministic gradient descent \cite{nesterov2018lectures}. Recall that for the iteration $\boldsymbol{x}_{k+1} = \boldsymbol{x}_k - \gamma(\nabla F(\boldsymbol{x}_k)+\nabla h(\boldsymbol{x}_k))$, if the step size satisfies $0 < \gamma \leq 2/(\tfrac{1}{n}\sum_{i=1}^{n} L_i +L_h+\lambda)$, then
\begin{equation*}
	\|\boldsymbol{x}_{k+1}-\boldsymbol{x}^*\|_2^2 \leq
	\Bigg( 1 - \frac{2\gamma \lambda (\frac{1}{n}\sum\limits_{i=1}^{n}L_i+L_h)}{\frac{1}{n}\sum\limits_{i=1}^{n}L_i+L_h+\lambda} \Bigg)^k
	\|\boldsymbol{x}_1-\boldsymbol{x}^*\|_2^2.
\end{equation*}
In particular, choosing the step size $\gamma=2/(\tfrac{1}{n}\sum_{i=1}^{n} L_i +L_h+\lambda)$ yields
\begin{equation}
	\|\boldsymbol{x}_{k+1}-\boldsymbol{x}^*\|_2^2 \leq
	\Bigg(\frac{\frac{1}{n}\sum\limits_{i=1}^{n}L_i+L_h-\lambda}{\frac{1}{n}\sum\limits_{i=1}^{n}L_i+L_h+\lambda}\Bigg)^{2k}
	\|\boldsymbol{x}_1-\boldsymbol{x}^*\|_2^2 . \label{eq:gd_opt}
\end{equation}

Contrasting \Cref{thm:1} with this classic result indicates that approximating $\nabla F(\boldsymbol{x})$ introduces two parameters, $\sigma_{\scriptscriptstyle F}$ and $s_{\scriptscriptstyle F}$. 
As shown in \eqref{eq:part4_mini}, these parameters stem from the bound of the variance of the stochastic gradient,
\begin{equation*}
	\frac{1}{B}\mathbb{E}\bigg[\bigg\|\frac{1}{\textstyle n p_{\scriptscriptstyle\xi_{k,1}}}\nabla f_{\xi_{k,1}}(\boldsymbol{x}_k)-\nabla F(\boldsymbol{x}_k)\bigg\|_{2}^{2}\bigg|\mathcal{F}_k\bigg]
	\leq s_{\scriptscriptstyle F}\|\nabla F\left(\boldsymbol{x}_k\right)+\nabla h\left(\boldsymbol{x}_k\right)\|_{2}^{2}+2\sigma_{\scriptscriptstyle F},
\end{equation*}
where
\begin{equation*}
	s_{\scriptscriptstyle F}=\frac{2}{B\lambda^2 n}\sum_{i=1}^{n} \frac{L_i^2}{n p_i},\; \sigma_{\scriptscriptstyle F}=\frac{1}{B}\Bigg(\frac{1}{ n}\sum_{i=1}^{n}\frac{1}{n p_i}\bigl\|\nabla f_i(\boldsymbol{x}^*)\bigr\|_2^2-\bigl\|\nabla F(\boldsymbol{x}^*)\bigr\|_2^2\Bigg).
\end{equation*}
Also, by definition, these two parameters depend mainly on the SG-term. The parameter $\sigma_{\scriptscriptstyle F}$ quantifies the variance of the stochastic gradient at the minimizer $\boldsymbol{x}^*$. This term arises from the stochasticity of sampling the component functions $f_i$, $i=1,\ldots, n$, capturing the inherent noise that persists even when the algorithm reaches the minimizer.
The parameter $s_{\scriptscriptstyle F}$ can be regarded as the relationship between the variance of the stochastic gradient and the full gradient $\nabla F\left(\boldsymbol{x}_k\right)+\nabla h\left(\boldsymbol{x}_k\right)$, indicating that the variance of the stochastic gradient  grows proportionally as the full gradient becomes larger.  

We then consider how $s_{\scriptscriptstyle F}$ and $\sigma_{\scriptscriptstyle F}$ affect the choice of step size. According to \eqref{eq:eta4} in \Cref{thm:1}, we have
\begin{equation*}
	0<\eta \leq \frac{2}{\bigl(s_{\scriptscriptstyle F} +1\bigr)\Big(\frac{1}{n}\sum\limits_{i=1}^{n} L_i + L_h+\lambda\Big)}.
\end{equation*}
Compared with deterministic gradient descent, the upper bound on the step size $\eta$ for the mini-batch SGD is additionally influenced by $s_{\scriptscriptstyle F}$. A larger $s_{\scriptscriptstyle F}$ reduces the allowable step size. In other words, we need to adopt a conservative step size to compensate for the uncertainty.

To further investigate the impact of the stochastic approximation, we present the following corollary.

\begin{corollary}\label{cor:1}
	Under the setting of \Cref{thm:1}, choosing
	\begin{equation}\label{eq:eta_opt}
		\eta=\bar{\eta}:=\frac{2}{\left(s_{\scriptscriptstyle F}+1\right)\Big(\frac{1}{n}\sum\limits_{i=1}^{n}L_i+L_h+\lambda\Big)}
	\end{equation}
	maximizes the rate
	\begin{equation}\label{eq:q_opt}
		q=\bar{q}:=\frac{4\lambda\Big(\frac{1}{n}\sum\limits_{i=1}^{n}L_i+L_h\Big)}{\left(s_{\scriptscriptstyle F}+1\right)\Big(\frac{1}{n}\sum\limits_{i=1}^{n} L_i + L_h+\lambda \Big)^{2}},
	\end{equation}
	and the radius becomes
	\begin{equation}\label{eq:R_opt}
		R=\bar{R}:=\frac{2\sigma_{\scriptscriptstyle F}}{\lambda\left(s_{\scriptscriptstyle F}+1\right)\Big(\frac{1}{n}\sum\limits_{i=1}^{n}L_i+L_h\Big)} .
	\end{equation}
	The constants $s_{\scriptscriptstyle F}$ and $\sigma_{\scriptscriptstyle F}$ are defined in~\eqref{eq:s_F} and~\eqref{eq:sigma_F}, respectively.
\end{corollary}

Based on \Cref{cor:1}, we discuss the impact of the SG-term on the convergence rate. Notice that the parameter $s_{\scriptscriptstyle F}$ appears in the denominator of $\bar{q}$ in \eqref{eq:q_opt}. Consequently, a larger $s_{\scriptscriptstyle F}$ results in a smaller achievable convergence rate $\bar{q}$, leading to a slower convergence process.

Moreover, the SG-term affects the radius of the convergence region. According to \eqref{eq:R_opt}, the radius $\bar{R}$ is proportional to $\sigma_{\scriptscriptstyle F}$. The denominator, which involves the convexity parameter and the Lipschitz constant, is strictly positive. So these terms only scale the magnitude of the radius. In contrast, $\sigma_{\scriptscriptstyle F}$ represents the inherent noise. It is the only parameter that determines whether the radius vanishes. Consequently, $\sigma_{\scriptscriptstyle F}$ determines whether the algorithm converges to the minimizer or just within its neighborhood.

Finally, \eqref{eq:s_F} shows that $s_{\scriptscriptstyle F}$ decreases when the batch size $B$ increases, meaning that the stochastic noise is lowered. Hence a larger step size can be used according to \eqref{eq:eta_opt}, leading to a faster rate in~\eqref{eq:q_opt}.  
At the same time, \eqref{eq:sigma_F} and~\eqref{eq:R_opt} indicate that $\sigma_{\scriptscriptstyle F}$ shrinks and thus $\bar{R}$ also shrinks as $B$ increases, so a larger batch size yields a solution which concentrates more around the minimizer $\boldsymbol{x}^*$. Letting $B\to\infty$ in \Cref{cor:1} yields
\begin{align*}
	\bar{\eta}\to\frac{2}{\frac{1}{n}\sum\limits_{i=1}^{n} L_i +L_h+\lambda},\quad\bar{q}(L_1,\ldots,L_n,L_h)\to\frac{4\lambda\Big(\frac{1}{n}\sum\limits_{i=1}^{n}L_i+L_h\Big)}
	{\Big(\frac{1}{n}\sum\limits_{i=1}^{n} L_i +L_h+\lambda\Big)^2},\quad \bar{R}\to0,
\end{align*}
which leads to
\[
\mathbb{E}\bigl[\|\boldsymbol{x}_{K+1}-\boldsymbol{x}^*\|_2^2\bigr]\leq
\Bigg(\frac{\frac{1}{n}\sum\limits_{i=1}^{n}L_i+L_h-\lambda}{\frac{1}{n}\sum\limits_{i=1}^{n}L_i+L_h+\lambda}\Bigg)^{2K}
\|\boldsymbol{x}_1-\boldsymbol{x}^*\|_2^2.
\]
Therefore, \Cref{thm:1} asymptotically matches the classic linear convergence rate of the gradient descent \eqref{eq:gd_opt} in expectation.

\section{Asymmetric Roles of Stochastic and Deterministic Gradients}\label{sec:4}
Building on the linear convergence result in \Cref{sec:3}, we now explain more explicitly how the stochastic approximation of $\nabla F(\boldsymbol{x})$ and the deterministic computation of $\nabla h(\boldsymbol{x})$ affect the convergence of the algorithm differently. 


For reference, we recall a classic analysis for SGD. Theorem 5.8 in \cite{garrigos2023handbook} (a simplified version of Theorem 3.1 in \cite{gower2019sgd}) considers the mini-batch SGD with $B=1$ and uniform sampling probabilities $p_{j}=1/n$. It states that if the step size satisfies $0<\eta\leq \frac{1}{2(\max_{1\leq i\leq n} L_i+L_h)}$, then the sequence $\{\boldsymbol{x}_k\}_{k=1}^{K+1}$ satisfies
\begin{align*}
	\mathbb{E}\left[\|\boldsymbol{x}_{K+1}-\boldsymbol{x}^*\|_2^{2}\right]
	\leq \bigl(1-\lambda\eta\bigr)^{K}\|\boldsymbol{x}_1-\boldsymbol{x}^*\|_2^{2}+\frac{2\eta}{\lambda n}\sum_{i=1}^{n}\|\nabla f_i(\boldsymbol{x}^*)+\nabla h(\boldsymbol{x}^*)\|_2^{2}.
\end{align*}

Notice that \Cref{thm:1} and the result above describe the linear convergence behavior of a same sequence when the step size $\eta$ is fixed and satisfies $0<\eta\leq \min\Big\{\frac{2}{(s_{\scriptscriptstyle F} +1)(\tfrac{1}{n}\sum_{i=1}^{n} L_i +L_h+\lambda)},\frac{1}{2(\max_{1\leq i\leq n} L_i+L_h)}\Big\}$, but they lead to different rate and different radius of the convergence region theoretically. Specifically, denote the respective rate and radius in \Cref{thm:1} as
\begin{equation*}
	q_1=\frac{2 \eta \lambda\Big(\frac{1}{n} \sum\limits_{i=1}^n L_i+L_h\Big)}{\frac{1}{n} \sum\limits_{i=1}^n L_i+L_h+\lambda}, \quad R_1=\frac{2 \eta^2 \sigma_F}{q_1}.
\end{equation*}
For the result of Theorem 5.8 in \cite{garrigos2023handbook}, we denote
\begin{equation*}
	q_2=\eta \lambda, \quad R_2=\frac{2\eta}{\lambda n}\sum_{i=1}^{n}\|\nabla f_i(\boldsymbol{x}^*)+\nabla h(\boldsymbol{x}^*)\|_2^{2}.
\end{equation*}
Then we have the following theorem.

\begin{theorem}\label{thm:41}
	Letting $B=1$, $p_j=1/n$, $j=1,\ldots,n$, and $$0<\eta\leq \min\Bigg\{\frac{2}{(s_{\scriptscriptstyle F} +1)\Big(\frac{1}{n}\sum\limits_{i=1}^{n} L_i + L_h+\lambda\Big)},\frac{1}{2\Big(\max\limits_{1\leq i\leq n} L_i+L_h\Big)}\Bigg\},$$ then the inequalities
	\begin{equation*}
		q_1 \geq q_2, \quad R_1 \leq R_2
	\end{equation*}
	hold. Furthermore, if $\tfrac{1}{n}\sum_{i=1}^n L_i+L_h>\lambda$ and $\sigma_{\scriptscriptstyle F}>0$, then the strict inequalities $q_1 > q_2$ and $R_1 < R_2$ hold.
\end{theorem}

\begin{proof}
	Observe that $q_1\geq q_2$ is equivalent to 
	\begin{align*}
		\frac{2 \eta \lambda\Big(\frac{1}{n} \sum\limits_{i=1}^n L_i+L_h\Big)}{\frac{1}{n} \sum\limits_{i=1}^n L_i+L_h+\lambda}\geq \eta \lambda 
		\Longleftrightarrow  \frac{1}{n} \sum_{i=1}^n L_i + L_h \geq \lambda,
	\end{align*}
	which holds by \Cref{def:Lip_con} and inequality \eqref{eq:str-cvx7}. Consequently, the strict inequality $q_1>q_2$ holds if and only if $\tfrac{1}{n} \sum_{i=1}^n L_i + L_h > \lambda$.
	For the radius of the convergence region, since $\nabla F(\boldsymbol{x}^*)+\nabla h(\boldsymbol{x}^*)=\nabla \psi(\boldsymbol{x}^*)=\boldsymbol{0}$, we have $\frac{1}{n} \sum_{i=1}^{n}\|\nabla f_i(\boldsymbol{x}^*)+\nabla h(\boldsymbol{x}^*)\|_2^{2}=\frac{1}{n}\sum_{i=1}^{n}\|\nabla f_i(\boldsymbol{x}^*)\|_2^2-\|\nabla F(\boldsymbol{x}^*)\|_2^{2}=\sigma_{\scriptscriptstyle F}$. Since $R_2$ can be written as $\frac{2\eta^2}{q_2n}\sum_{i=1}^{n}\|\nabla f_i(\boldsymbol{x}^*)+\nabla h(\boldsymbol{x}^*)\|_2^{2}$ and $q_1\geq q_2$, we have $R_1$ $\leq R_2$. Likewise, the strict inequality $q_1 > q_2$ implies $R_1< R_2$ when $\sigma_{\scriptscriptstyle F}>0$. 
\end{proof}

\Cref{thm:41} shows that our analysis gives a faster linear convergence rate and a smaller radius of the convergence region theoretically for the same algorithm sequence. This difference comes from the separate handling of the SG-term and DG-term in our proof.

To deepen our understanding of the algorithm, we compare the best possible convergence guarantees derived from our analysis with those from Theorem 5.8 in \cite{garrigos2023handbook}. Recall that \Cref{cor:1} defines $\bar{\eta}$ as the step size that maximizes our convergence rate to $\bar{q}$ with a corresponding $\bar{R}$. 
Similarly, we denote $\tilde{\eta}$ as the maximum step size permitted by Theorem 5.8 in \cite{garrigos2023handbook} and let $\tilde{q}$ and $\tilde{R}$ represent the resulting rate and radius, that is,
\begin{align}
	&\tilde{\eta}=\frac{1}{2\Big(\max\limits_{1\leq i\leq n} L_i+L_h\Big)}, \tilde{q}=\frac{\lambda}{2\Big(\max\limits_{1\leq i \leq n}{L_i}+L_h\Big)},\nonumber \\
	&\tilde{R}=\frac{1}{\lambda \Big(\max\limits_{1\leq i\leq n} L_i+L_h\Big) n}\sum_{i=1}^{n}\|\nabla f_i(\boldsymbol{x}^*)+\nabla h(\boldsymbol{x}^*)\|_2^{2}.\label{eq:q_tilde}
\end{align}

By writing these expressions as functions of $\{L_i\}_{i=1}^{n}$ and $L_h$, we observe that the Lipschitz constants appear asymmetrically in our results, unlike in the classic analysis. This asymmetry reflects the different influence of $f_i(\boldsymbol{x})$, $i=1,\ldots,n$, and $h(\boldsymbol{x})$, as stated more precisely in the following theorem.

\begin{theorem}\label{thm:different_impact}
	Let $B=1$, $p_j=1/n$, $j=1,\ldots,n$. Given a constant $0<c$ $<$ $\min\{L_1,$ $\ldots,$ $L_n,$ $L_h\}$, we have
	\begin{align*}
		&\bar{\eta}(L_1-c,\ldots,L_n-c,L_h)\;>\;\bar{\eta}(L_1,\ldots,L_n,L_h-c),\\
		&\bar{q}(L_1-c,\ldots,L_n-c,L_h)\;>\;\bar{q}(L_1,\ldots,L_n,L_h-c),\\
		&\tilde{\eta}(L_1-c,\ldots,L_n-c,L_h)=\tilde{\eta}(L_1,\ldots,L_n,L_h-c).\\
		&\tilde{q}(L_1-c,\ldots,L_n-c,L_h)=\tilde{q}(L_1,\ldots,L_n,L_h-c).
	\end{align*}
\end{theorem}

\begin{proof}
	Observe that
	\begin{equation*}
		\frac{\bar{\eta}(L_1-c,\ldots,L_n-c,L_h)}{\bar{\eta}(L_1,\ldots,L_n,L_h-c)}=\frac{\frac{2}{B\lambda^2 n}\sum\limits_{i=1}^{n}\frac{L_i^2}{np_i}+1}{\frac{2}{B\lambda^2 n}\sum\limits_{i=1}^{n}\frac{(L_i-c)^2}{n p_i}+1}>1,
	\end{equation*}
	and similarly,
	\begin{equation*}
		\frac{\bar{q}(L_1-c,\ldots,L_n-c,L_h)}{\bar{q}(L_1,\ldots,L_n,L_h-c)}=\frac{\frac{2}{B\lambda^2 n}\sum\limits_{i=1}^{n}\frac{L_i^2}{np_i}+1}{\frac{2}{B\lambda^2 n}\sum\limits_{i=1}^{n}\frac{(L_i-c)^2}{n p_i}+1}>1,
	\end{equation*}
	which establish the strict inequalities. The last two equalities 
	$\tilde{\eta}(L_1-c,\ldots,L_n-c,L_h)=\tilde{\eta}(L_1,\ldots,L_n,L_h-c)$ and $\tilde{q}(L_1-c,\ldots,L_n-c,L_h)=\tilde{q}(L_1,\ldots,L_n,L_h-c)$ follow directly from \eqref{eq:q_tilde}.
\end{proof}

Let $L_F=\tfrac{1}{n}\sum_{i=1}^{n}L_i$ denote the Lipschitz constant for the gradient of $F(\boldsymbol{x})$. Consider two scenarios where we lower every $L_i$ by $c$ or alternatively lower $L_h$ by $c$. Both cases reduce the total smoothness $L_F+L_h$ by the same amount. Under the classic framework, these two changes yield an identical step size $\tilde{\eta}$ and convergence rate $\tilde{q}$. 

In our framework, however, the step size $\bar{\eta}$ and rate $\bar{q}$ are sensitive to the structure of the objective function through the parameter $s_{\scriptscriptstyle F}$. Specifically, $s_{\scriptscriptstyle F}$ is reduced when $L_i$ decreases, but remains constant if only $L_h$ decreases. As a result, improving the smoothness of the SG-term leads to a larger step size and a faster convergence rate compared to an equivalent improvement in DG-term. This difference arises just because the gradients of the SG-term $F(\boldsymbol{x})$ and the DG-term $h(\boldsymbol{x})$ are handled differently by the algorithm, which is captured by our analysis.

Next, we give a simple condition where our step size bound in \Cref{thm:1} is less restrictive than the bound in Theorem 5.8 of \cite{garrigos2023handbook}. Under the same condition, the best linear convergence rate obtained by our analysis is also larger.

\begin{theorem}\label{thm:our_faster}
	Let $B=1$ and $p_{j}=1/n$ for $j=1,\ldots,n$.
	If $\tfrac{2}{n}\sum_{i=1}^{n}L_i^{2}\le \lambda^{2}$,
	then $\tilde{\eta}\leq \bar{\eta}$ and $\tilde{q}\leq \bar{q}$.
\end{theorem}

\begin{proof}
	For $\tilde{\eta}\leq \bar{\eta}$, we need to show
	\begin{equation*}
		\frac{1}{2\Big(\max\limits_{1\leq i\leq n} L_i + L_h\Big)} \leq \frac{2}{\Big(\frac{2}{\lambda^2 n}\sum\limits_{i=1}^{n}L_i^2+1\Big)\Big(\frac{1}{n} \sum \limits_{i=1}^{n}L_i +L_h + \lambda\Big)},
	\end{equation*}
	which is equivalent to proving
	\begin{equation*}
		\left(\frac{2}{n}\sum\limits_{i=1}^{n}L_i^2+\lambda^2\right)\left(\frac{1}{n} \sum \limits_{i=1}^{n}L_i +L_h + \lambda\right) \leq 4 \lambda^2 \left(\max\limits_{1\leq i\leq n} L_i + L_h\right).
	\end{equation*}
	Since $\tfrac{2}{n}\sum_{i=1}^{n}L_i^{2}\le \lambda^{2}$ and $\lambda\leq \frac{1}{n}\sum\limits_{i=1}^{n} L_i + L_h$, we have
	\begin{align*}
		\left( \frac{2}{n}  \sum_{i=1}^{n}  L_i^2 + \lambda^2 \right)  \left( \frac{1}{n}  \sum_{i=1}^{n}  L_i + L_h + \lambda \right) &\leq 4\lambda^2  \left( \frac{1}{n}  \sum_{i=1}^{n}  L_i + L_h \right)\\
		 &\leq 4\lambda^2  \left( \max_{1\leq i\leq n}  L_i + L_h \right)
	\end{align*}
	
	For $\tilde{q}\leq \bar{q}$, we need to show 
	\begin{align*}
		\frac{\lambda}{2\Big(\max\limits_{1\leq i\leq n}\{L_i\}+L_h\Big)}
		\leq \frac{4\lambda\Big(\frac{1}{n}\sum\limits_{i=1}^nL_i + L_h\Big)}{\Big(\frac{1}{n} \sum \limits_{i=1}^{n}L_i +L_h + \lambda\Big)^2\Big(\frac{2}{\lambda^2 n}\sum\limits_{i=1}^{n}L_i^2+1\Big)}.
	\end{align*}
	This is equivalent to proving
	\begin{align}
		&\left(\frac{1}{n}\sum_{i=1}^nL_i+L_h+\lambda\right)^{2}\left(\frac{2}{n}\sum_{i=1}^nL_i^2+\lambda^2\right)\nonumber \\
		\leq&8\left(\frac{1}{n}\sum_{i=1}^{n}L_i+L_h\right)\left(\max_{1\leq i\leq n}\{L_i\}+L_h\right)\lambda^2.\label{eq:cor3_ineq1}
	\end{align}
	Since $\lambda\leq \frac{1}{n}\sum_{i=1}^{n} L_i + L_h$ and $\max\limits_{1\leq i \leq n}\{L_i\}\geq\frac{1}{n}\sum\limits_{i=1}^n L_i$, we have
	\begin{align*}
		&\frac{2}{n}\sum\limits_{i=1}^n L_i^2\leq \lambda^2\\
		\Longleftrightarrow& \;4\left(\frac{1}{n}\sum_{i=1}^n L_i + L_h\right)^{2}\left(\frac{2}{n}\sum\limits_{i=1}^n L_i^2 +\lambda^2\right)\leq 8\left(\frac{1}{n}\sum_{i=1}^n L_i + L_h\right)^{2} \lambda^2\Longrightarrow \eqref{eq:cor3_ineq1}.
	\end{align*}
	The proof is complete. 
\end{proof}

Since $s_{\scriptscriptstyle F}
=\tfrac{2}{\lambda^{2}n}\sum_{i=1}^{n}L_i^{2}$,
the condition in \Cref{thm:our_faster} is equivalent to $s_{\scriptscriptstyle F}\le 1$. Recall that 
the parameter $s_{\scriptscriptstyle F}$ characterizes the proportional growth of the stochastic gradient with respect to the full gradient.
Intuitively, we treat the SG-term and the DG-term separately in our analysis, where the DG-term is highlighted. This is why our analysis admits a larger range of step sizes that still yields linear convergence, and why the largest achievable rate is also improved, when $s_{\scriptscriptstyle F}$ is small.

From an alternative perspective, \eqref{eq:L-smooth1} and \eqref{eq:str-cvx1} indicate that $\lambda \le \tfrac{1}{n}\sum_{i=1}^{n} L_i + L_h$, whereas \Cref{thm:our_faster} requires $\tfrac{2}{n}\sum_{i=1}^{n}L_i^{2}\le\lambda^{2}$. Therefore, the condition in \Cref{thm:our_faster} becomes easily checkable when the convexity parameter $\lambda$ of $\psi(\boldsymbol{x})$ is dominated by $h(\boldsymbol{x})$. For example, if $h(\boldsymbol{x})=\tfrac{L_h}{2}\|\boldsymbol{x}\|_2^{2}$, then $\lambda\ge L_h$, and the condition in \Cref{thm:our_faster} reduces to $\tfrac{2}{n}\sum_{i=1}^{n}L_i^{2}\le L_h^{2}$.
Intuitively, the DG-term now provides enough strong convexity to mitigate the noise caused by stochastic gradients, which allows a larger step size and faster convergence.

Although \Cref{thm:our_faster} compares the largest admissible step size and the best linear rate of the two analyses, it leaves open if the convergence radius of our result is smaller when the optimal step size is adopted for each analysis. We now present a sufficient condition under which our framework yields superior guarantees in terms of both the linear rate and the convergence radius.

We focus on $B=1$ and uniform sampling $p_j=1/n$, $j=1,\ldots,n$. According to \eqref{eq:q_opt}, \eqref{eq:R_opt}, and \eqref{eq:q_tilde}, the inequalities $\bar{q}\ge \tilde{q}$ and $\bar{R}\le \tilde{R}$ are equivalent to
\begin{equation}\label{eq:sandwich4}
	\frac{2\Big(\max\limits_{1\leq i \leq n} L_i+L_h\Big)}{L_F+L_h}\le 1+\frac{2}{\lambda^2 n}\sum_{i=1}^{n}L_i^2
	\le \frac{8\Big(\max\limits_{1\leq i \leq n} L_i+L_h\Big)(L_F+L_h)}{(L_F+L_h+\lambda)^2},
\end{equation}
where $L_F=\tfrac{1}{n} \sum_{i=1}^n L_i$.
To clearly show when \eqref{eq:sandwich4} holds, we first define
\[
\lambda_{+}=\sqrt{\frac{\bigl(\frac{2}{n}\sum_{i=1}^{n}L_i^2\bigr)(L_F+L_h)}
	{2(\max\limits_{1\leq i \leq n} L_i+L_h)-(L_F+L_h)}}.
\]
Moreover, let $\lambda_{-}$ be the unique number in the interval $0<\lambda\le L_F+L_h$ that satisfies
\begin{equation}\label{eq:lambdamin4}
	1+\frac{2}{\lambda^2 n}\sum_{i=1}^{n}L_i^2
	=\frac{8\Big(\max\limits_{1\leq i \leq n} L_i+L_h\Big)(L_F+L_h)}{(L_F+L_h+\lambda)^2},
\end{equation}
when
\begin{equation}
	\frac{2}{n}\sum_{i=1}^{n}L_i^{2}\le (L_F+L_h)\left(2\left(\max\limits_{1\leq i \leq n} L_i+L_h\right)-(L_F+L_h)\right).\label{eq:compatibility}
\end{equation}

The following proof of \Cref{thm:sandwich4} shows these quantities are well-defined.

\begin{theorem}\label{thm:sandwich4}
	Let $B=1$ and $p_j=1/n$ for $j=1,\ldots,n$. The inequalities $\bar{q}\ge \tilde{q}$ and $\bar{R}\le \tilde{R}$ hold if $\frac{2}{n}\sum_{i=1}^{n}L_i^{2}\le (L_F+L_h)\bigl(2(\max_{1\leq i \leq n} L_i+L_h)-(L_F+L_h)\bigr)$ and $\lambda_{-}\le \lambda\le \lambda_{+}$.
\end{theorem}

\begin{proof}
	As shown above, the conditions $\bar{q}\ge \tilde{q}$ and $\bar{R}\le \tilde{R}$ are equivalent to \eqref{eq:sandwich4}. We therefore study when \eqref{eq:sandwich4} holds. Since $\psi(\boldsymbol{x})$ is $(L_F+L_h)$-smooth and $\lambda$-strongly convex and according to \Cref{def:Lip_con} and inequality \eqref{eq:str-cvx7}, we have $\lambda\le L_F+L_h$, so it suffices to work on the interval $0<\lambda\le L_F+L_h$.
	
	The left inequality in \eqref{eq:sandwich4} is equivalent to
	\[
	\frac{2}{\lambda^2 n}\sum_{i=1}^{n}L_i^2 \ge \frac{2\Big(\max\limits_{1\leq i \leq n} L_i+L_h\Big)}{L_F+L_h}-1.
	\]
	Since $\max_{1\leq i \leq n} L_i\ge L_F$, the quantity $2(\max_{1\leq i \leq n} L_i+L_h)-(L_F+L_h)$ is positive, and the above inequality is equivalent to $\lambda\le \lambda_{+}$.
	
	For the right inequality in \eqref{eq:sandwich4}, rewrite it as
	\[
	\frac{2}{n}\sum_{i=1}^{n}L_i^2
	\le \lambda^2\Bigg(\frac{8\Big(\max\limits_{1\leq i \leq n} L_i+L_h\Big)(L_F+L_h)}{(L_F+L_h+\lambda)^2}-1\Bigg).
	\]
	Define
	\[
	\phi(\lambda)=\lambda^2\Bigg(\frac{8\Big(\max\limits_{1\leq i \leq n} L_i+L_h\Big)(L_F+L_h)}{(L_F+L_h+\lambda)^2}-1\Bigg),
	\qquad 0<\lambda\le L_F+L_h.
	\]
	Differentiating gives
	\[
	\phi'(\lambda)=2\lambda\Bigg(\frac{8\Big(\max\limits_{1\leq i \leq n} L_i+L_h\Big)(L_F+L_h)^2}{(L_F+L_h+\lambda)^3}-1\Bigg).
	\]
	For $0<\lambda<L_F+L_h$, we have $(L_F+L_h+\lambda)^3<8(L_F+L_h)^3$, which implies $\phi'(\lambda)>0$. Hence $\phi$ is strictly increasing on $0<\lambda\le L_F+L_h$. Moreover,
	\[
	\lim_{\lambda\to 0^{+}}\phi(\lambda)=0,\qquad
	\phi(L_F+L_h)=(L_F+L_h)\left(2\left(\max\limits_{1\leq i \leq n} L_i+L_h\right)-(L_F+L_h)\right).
	\]
	Consequently, the right inequality in \eqref{eq:sandwich4} can hold if
	\[
	\frac{2}{n}\sum_{i=1}^{n}L_i^{2}\le (L_F+L_h)\left(2\left(\max\limits_{1\leq i \leq n} L_i+L_h\right)-(L_F+L_h)\right),
	\]
	which is exactly \eqref{eq:compatibility}. Under this condition, the strict monotonicity of $\phi$ ensures a unique $\lambda_{-}\in (0, L_F+L_h]$ such that $\phi(\lambda_{-})=\frac{2}{n}\sum_{i=1}^{n}L_i^{2}$, which is exactly \eqref{eq:lambdamin4}. The right inequality in \eqref{eq:sandwich4} then holds when $\lambda\ge \lambda_{-}$.
	
	Combining the two parts shows that \eqref{eq:sandwich4} holds if \eqref{eq:compatibility} holds and $\lambda_{-}\le \lambda\le \lambda_{+}$. It remains to show that the interval is nonempty when \eqref{eq:compatibility} holds. First, \eqref{eq:compatibility} implies that $\lambda_{+}\le L_F+L_h$. If $\lambda_{+}<L_F+L_h$, then
	\begin{align*}
		&\phi(\lambda_{+})-\Bigg(\frac{2\Big(\max\limits_{1\leq i \leq n} L_i+L_h\Big)}{L_F+L_h}-1\Bigg)\lambda_{+}^2 \\
		=&\frac{2\Big(\max\limits_{1\leq i \leq n} L_i+L_h\Big)\lambda_{+}^2\bigl(L_F+L_h-\lambda_{+}\bigr)\bigl(3(L_F+L_h)+\lambda_{+}\bigr)}{(L_F+L_h)(L_F+L_h+\lambda_{+})^2}>0,
	\end{align*}
	which implies $\phi(\lambda_{+})>\frac{2}{n}\sum_{i=1}^{n}L_i^{2}$ and hence $\lambda_{-}<\lambda_{+}$. If $\lambda_{+}=L_F+L_h$, then \eqref{eq:compatibility} holds with equality and $\phi(\lambda_{+})=\frac{2}{n}\sum_{i=1}^{n}L_i^{2}$. So the definition of $\lambda_{-}$ and the strict monotonicity of $\phi$ gives $\lambda_{-}=\lambda_{+}$. This completes the proof. 
\end{proof}

A simple sufficient rule for \eqref{eq:compatibility} is $L_h\ge (\sqrt{2}-1)\max_{1\leq i \leq n} L_i$. Indeed, $L_i^2\le \max_{1\leq j \leq n} \{L_j\}L_i$ yields $\frac{2}{n}\sum_{i=1}^{n}L_i^{2}\le 2\max_{1\leq i \leq n} \{L_i\}L_F$, and we have

\begin{equation*}
	(L_F+L_h)\Bigl(2\bigl(\max_{1\le i\le n}L_i+L_h\bigr)-(L_F+L_h)\Bigr)=
	L_h^2+2\max_{1\le i\le n}\{L_i\}L_h+2\max_{1\le i\le n}\{L_i\}L_F-L_F^2.
\end{equation*}

Under the bound $L_h\ge (\sqrt{2}-1)\max_{1\leq i \leq n} L_i$, the inequality $$L_h^2+2\max\limits_{1\leq i \leq n} \{L_i\}L_h\ge \max\limits_{1\leq i \leq n} L_i^2\ge L_F^2$$holds, and \eqref{eq:compatibility} follows.

\section{Numerical Experiments}\label{sec:5}
This section presents numerical experiments to support our theoretical results and provide further discussion. We run all tests in MATLAB R2023b on a desktop computer with 64 GB RAM and an Intel Core i9-12900K CPU.

The logistic regression problem discussed in \Cref{sec:2.3},
\begin{equation*}
	\min_{\boldsymbol{x} \in \mathbb{R}^d}
	\frac{1}{n}\sum_{i=1}^{n}\log\bigl(1+\exp(-b_i \boldsymbol{a}_i^{\top}\boldsymbol{x})\bigr)
	+\frac{L_h}{2}\|\boldsymbol{x}\|_2^2,
\end{equation*}
is selected as a concrete instance of the problem \eqref{eq:opt1}, that is, $f_i(\boldsymbol{x})=\log(1+\exp(-b_i \boldsymbol{a}_i^{\top} \boldsymbol{x}))$,  $i=1,\ldots,n$, and $ h(\boldsymbol{x})=\frac{L_h}{2}\|\boldsymbol{x}\|_2^2$. We set $n = 10^{6}$ and $d = 10^{2}$. The Lipschitz constants are $L_i=\|\boldsymbol{a}_i\|_2^2/4$ for $\nabla f_i$ and $L_h$ for $\nabla h$. The convexity parameter of $\psi(\boldsymbol{x})$ is $\lambda = L_h$. Following \cite{gower2019sgd}, we independently draw every element of each $\boldsymbol{a}_i$ from the standard normal distribution, and each $b_i$ from $\{1,-1\}$ uniformly.  
We use uniform sampling $p_i = 1/n$ and set $\boldsymbol{x}_1 = (10,\dots,10)^{\top}$.  
The optimal solution $\boldsymbol{x}^{*}$ is precomputed via gradient descent until $\|\nabla\psi(\boldsymbol{x}^{*})\|_2 \le 10^{-12}$. Results are averaged over one hundred runs.

\subsection{Different impacts of the SG-term and DG-term on convergence}

\subsubsection{Convergence rate}\label{sec:4.1.1}

The Lipschitz constant for the gradient of $F(\boldsymbol{x})$ is $L_F = \tfrac{1}{n}\sum_{i=1}^{n} L_i$, which equals 25.002 (rounded to the nearest thousandth) in this instance. To show how the SG-term $F(\boldsymbol{x})$ and DG-term $h(\boldsymbol{x})$ affect the convergence of the algorithm differently, we test $L_h = 100, 50, 25, 10,$ and $5$ according to the scale of $L_F$. These settings give three scenarios:  
(a) when $L_h$ is $100$ or $50$, we have $L_h > L_F$, so the DG-term $h(\boldsymbol{x})$ provides sufficient strong convexity and its gradient $\nabla h(\boldsymbol{x})$ serves as the main component of the descent direction;  
(b) when $L_h = 25$, the Lipschitz constant for the gradient of the SG-term and the convexity parameter of the DG-term are roughly equal;  
(c) when $L_h$ equals $10$ or $5$, we have $L_h < L_F$, and the noise introduced by the stochastic gradient of the SG-term $F(\boldsymbol{x})$ has a greater impact on the algorithm. In addition, set the batch size to $B = 1$ and denote the step size proposed in \Cref{cor:1} as $\eta_1 := \frac{2}{\left(s_{\scriptscriptstyle F}+1\right)( \tfrac{1}{n}\sum_{i=1}^{n} L_i + L_h+\lambda)}$, and the step size in \cite{garrigos2023handbook} as $\eta_2 := \frac{1}{2(\max_{1\leq i\leq n}\{L_i\} + L_h)}$.

We compare these two step sizes in \Cref{pic:compare}. The hollow-circle blue curve and the solid-circle green curve correspond to $\eta_1$ and $\eta_2$, respectively. We plot the relative error $\tfrac{\|\boldsymbol{x}_k-\boldsymbol{x}^*\|_2}{\|\boldsymbol{x}_1-\boldsymbol{x}^*\|_2}$ over iterations. Each curve initially shows linear decay, confirming linear convergence. As the iterations proceed, each curve flattens out to an almost horizontal line, showing that the error no longer decreases and the algorithm has entered a neighborhood of the minimizer. These match \Cref{thm:1} and \Cref{cor:1}.

\begin{figure}[H]
	\centering
	\includegraphics[width=1\textwidth]{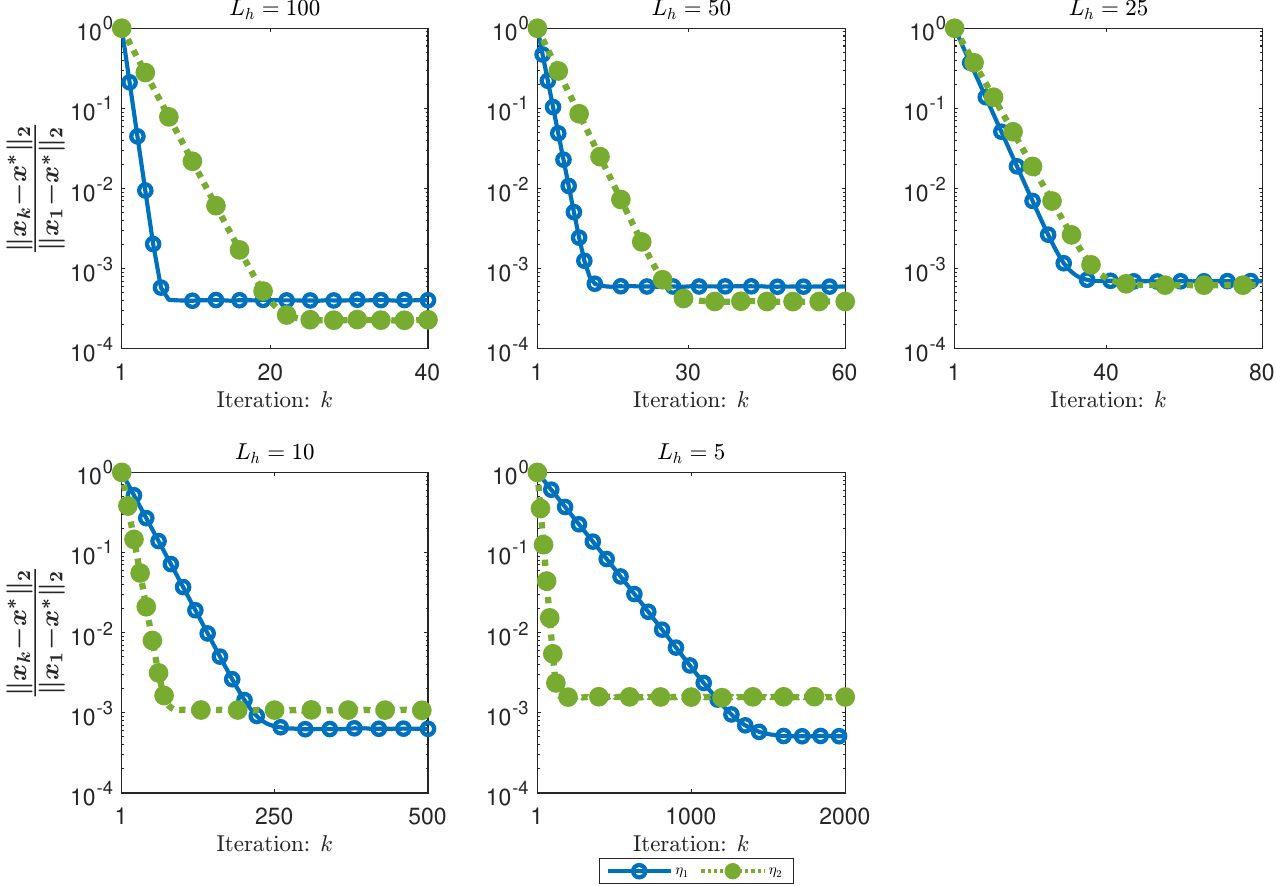}
	\caption{Performance comparison of mini-batch SGD under step sizes $\eta_1$ and $\eta_2$ on problems with different $L_h$.}\label{pic:compare}
\end{figure}

When $L_h$ is large, the first two subplots in \Cref{pic:compare} show that the linear part of the hollow-circle blue curve is steeper than the solid-circle green curve. Hence the linear convergence rate is higher with $\eta_1$, in agreement with \Cref{thm:our_faster}. The proposed step size takes advantage of the sufficient strong convexity provided by the DG-term and therefore attains faster convergence. When $L_h = 25$, namely $L_h \approx L_F$, we have $\tfrac{2}{n}\sum_{i=1}^{n} L_i^2 = 1275.233$ (rounded to the nearest thousandth) and $\lambda^2 = L_h^2 = 625$ by direct computation, so the condition $\tfrac{2}{n}\sum_{i=1}^{n} L_i^2 \le \lambda^2$ in \Cref{thm:our_faster} does not hold. However, the third subplot of  \Cref{pic:compare} indicates that the linear convergence rate with $\eta_1$ remains faster. The reason is that the condition in \Cref{thm:our_faster} is sufficient but not necessary. Thus, the step size $\eta_1$ may still improve convergence rate even when the DG-term provides less strong convexity than the corollary requires.

\subsubsection{Radius of Convergence Region}
Next, we study how the Lipschitz constant for the gradient of the DG-term $h(\boldsymbol{x})$ affects the radius of the convergence region when the step size $\eta_1$ is used.  
In each experiment, we let the algorithm run for a sufficiently long time after the iteration point had entered the convergence region.  
The largest value of $\frac{\|\boldsymbol{x}_k-\boldsymbol{x}^*\|_2}{\|\boldsymbol{x}_1-\boldsymbol{x}^*\|_2}$ over the last $20\%$ of the iterations is taken as an approximation of the relative radius (namely the radius normalized by the initial value), denoted by $R_{\text{rel}}$. Table \ref{tab:1} reports $R_{\text{rel}}$ for a wider range of $L_h$.

\begin{table}[h]
	\centering
	\caption{Relative radius $R_{\text{rel}}$ ($\times 10^{-4}$) of convergence regions for different $L_h$.}\label{tab:1}
{
	\setlength{\tabcolsep}{5pt}
	\begin{tabular}{cccccccccc}
		\toprule
		$L_h$ & 1000 & 500 & 100 & 50 & 25 & 20 & 10 & 5 & 1 \\
		\midrule
		$R_{\text{rel}}$ ($\times 10^{-4}$) & 0.498 & 0.980 & 4.085 & 6.078 & 7.053 & 7.098 & 6.427 & 5.218 & 2.417 \\
		\bottomrule
	\end{tabular}
}

\end{table}

Overall, $R_{\text{rel}}$ reaches its maximum at $L_h=20$ in this experiment, and then follows two distinct trends.  
When $L_h>20$, $R_{\text{rel}}$ decreases as $L_h$ grows, indicating that
once the DG-term provides sufficient strong convexity, the noise introduced by stochastic gradients is well controlled and higher accuracy is achieved.  
When $L_h<20$, $R_{\text{rel}}$ also falls as $L_h$ decreases.  
In this range, the DG-term has less impact than the SG-term, and our step size becomes conservative, which again leads to a smaller radius.  
These observations indicate that the step size $\eta_1$, which sufficiently considers the impact of the DG-term, can guarantee the accuracy achieved by the algorithm to some extent.

Recalling the expression of $\bar{R}$ in \eqref{eq:R_opt}, we have $\lambda=L_h$ for logistic regression, so \eqref{eq:R_opt} reduces to
\[
\bar{R}=\frac{2\sigma_{\scriptscriptstyle F}}{\Big(\frac{2}{n^{2}}\sum\limits_{i=1}^{n}L_i^{2}\sum\limits_{i=1}^{n}L_i\Big)\frac{1}{L_h}+\frac{2}{n}\sum\limits_{i=1}^{n}L_i^{2}+\Big(\frac{1}{n}\sum\limits_{i=1}^{n}L_i\Big)L_h+L_h^{2}}
\]
in this experiment. With $\sigma_{\scriptscriptstyle F}$ and $L_i$, $i=1,\ldots,n$, fixed, the right-hand side can be seen as a function of $L_h$.  
Since $\sigma_{\scriptscriptstyle F}$, $L_i$, $i=1,\ldots,n$, and $L_h$ are non-negative, this function attains its maximum for some $L_h$, which partly explains the phenomenon observed above.

\subsection{Impact of Batch Sizes}
In this subsection, we set $L_h = 10$ and choose the step size according to \eqref{eq:eta_opt}.  
To show how the algorithm behaves when $B \to \infty$, we run experiments for $B = 10$, $100$, and $1000$, and compare the mini-batch SGD with the classic gradient descent.  
The results are presented in Fig. \ref{pic:batch}.

\begin{figure}[h]
	\centering
	\includegraphics[width=0.7\textwidth]{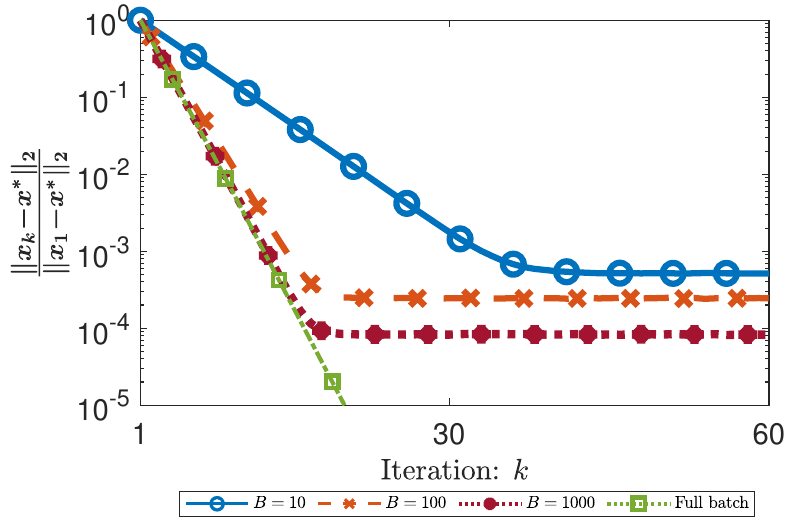}
	\caption{Convergence of classic gradient descent and mini-batch SGD under different batch sizes.}\label{pic:batch}
\end{figure}

In \Cref{pic:batch}, the green curve with square markers shows the decay of $\frac{\|\boldsymbol{x}_k - \boldsymbol{x}^*\|_2}{\|\boldsymbol{x}_1 - \boldsymbol{x}^*\|_2}$ using gradient descent with the step size $2/(\frac{1}{n}\sum_{i=1}^n L_i + L_h + \lambda)$ from \cite{nesterov2018lectures}.  
The remaining curves represent the mini-batch SGD with the three batch sizes.  
We observe that the mini-batch SGD converges to a region around $\boldsymbol{x}^*$ which shrinks as the batch size increases.  
Moreover, increasing $B$ causes the slope of the linear part of each mini-batch SGD curve to approach that of the gradient descent curve, showing that the convergence rate of the mini-batch SGD approaches the rate of the gradient descent when the batch size increases.  
These observations empirically confirm that \Cref{thm:1} asymptotically matches the classic linear convergence result of the gradient descent in expectation as $B \to \infty$.

\section{Conclusions}\label{sec:6}
Do the stochastic approximation of $\nabla F(\boldsymbol{x})$ and the deterministic computation of $\nabla h(\boldsymbol{x})$ impact the performance of  mini-batch SGD on the problem \eqref{eq:opt1} differently? We theoretically answer this question by capturing their asymmetric roles, under the assumptions that
$\psi(\boldsymbol{x})$ is strongly convex, with $\nabla f_i(\boldsymbol{x})$, $i=1,\dots,n$, and $\nabla h(\boldsymbol{x})$ being Lipschitz continuous. The main theorem proves the linear convergence of $\{\boldsymbol{x}_k\}$ to a neighborhood of the minimizer. The step size, the convergence rate, and the radius of the convergence region depend on the Lipschitz constants of $\nabla f_i(\boldsymbol{x})$, $i=1,\ldots,n$, and $\nabla h(\boldsymbol{x})$ asymmetrically, thereby showing their different influence. This analysis yields sharper theoretical guarantees, including a faster convergence rate and a smaller convergence region. An even faster rate is achieved  when $h(\boldsymbol{x})$ contributes enough strong convexity to $\psi(\boldsymbol{x})$. Numerical experiments support our theory.

Future work includes exploring whether the asymmetry between stochastic and deterministic gradient computations also arises under weaker smoothness assumptions, such as non-smooth settings. In addition, whether the similar asymmetric analysis could be applied to other algorithms?

\bibliographystyle{siamplain}
\bibliography{references}
\end{document}